\documentclass[a4paper]{article}

\usepackage{amsmath}
\usepackage{amssymb,amsthm} % put amsthm after amsmath for qedhere to work
\usepackage{bbm,enumitem}
\usepackage[mathscr]{euscript}
\usepackage{tikz-cd}
\usepackage[colorlinks=true]{hyperref} 

\hypersetup{
    colorlinks=true,
    linkcolor=[HTML]{0D47A1},		% internal link
    citecolor=[HTML]{0D47A1},		% bibliography
    urlcolor=[HTML]{0D47A1}		    % url
}
\usepackage[capitalize,nameinlink]{cleveref}
\usepackage{fullpage}

\usepackage{tabularx}
\usepackage{multirow}
\usepackage{float}

\usepackage[shortcuts]{extdash}
\newcommand\oo{$\infty$\=/}

\definecolor{mat}{HTML}{ffd6ad}

\definecolor{andre}{HTML}{a0ceff}
\definecolor{eric}{HTML}{ffadad}
\definecolor{georg}{HTML}{000000}

\numberwithin{equation}{subsection}% 

\newtheorem {theorem}[equation]{Theorem}
\newtheorem {corollary}[equation]{Corollary}
\newtheorem {proposition}[equation]{Proposition}
\newtheorem {lemma}[equation]{Lemma}
\theoremstyle{definition}
\newtheorem {definition}[equation]{Definition}
\newtheorem {construction}[equation]{Construction}
\newtheorem {example}[equation]{Example}
\newtheorem {examples}[equation]{Examples}
\newtheorem {remark}[equation]{Remark}

\newtheorem{thmintro}{Theorem}
\newenvironment{thm-intro}[1]
  {\renewcommand\thethmintro{#1}\thmintro\itshape}
  {\endthmintro}
\newenvironment{thm-introref}[2]
  {\renewcommand\thethmintro{#1}\thmintro(#2)\itshape}
  {\endthmintro}
\newtheorem{defnintro}{}
\newenvironment{defn-intro}[1]
  {\renewcommand\thedefnintro{#1}\defnintro\normalfont}
  {\enddefnintro}

\newcommand{\pp}{\,\Box\,}%     
\newcommand{\mc}[1]{\ensuremath{\mathcal{#1}}}%
\newcommand{\mb}[1]{\ensuremath{\mathbbm{#1}}}%
\newcommand{\scr}[1]{\ensuremath{\mathscr{#1}}}%
\newcommand{\xto}[1]{\ensuremath{\xrightarrow{#1}}}%
\newcommand{\cogap}{\operatorname*{cogap}}%
\newcommand{\colim}{\operatorname*{colim}}%
\newcommand{\id}{\operatorname{id}}%
\newcommand{\Id}{\mathrm{Id}}%
\newcommand{\op}{\ensuremath{\mathrm{op}}}%
\newcommand{\pr}{\ensuremath{\mathrm{pr}}}%
\newcommand{\ul}[1]{\underline{#1}}%
\newcommand{\ev}{\ensuremath{\mathrm{ev}}}%
\newcommand{\Fun}{\ensuremath{\mathrm{Fun}}}%
\newcommand{\Fin}{\ensuremath{\rm Fin}}%
\newcommand{\Finp}{\ensuremath{\mathrm{Fin}_{\ast}}}%
\newcommand{\map}{\ensuremath{\mathrm{map}}}%  % 
\newcommand{\nat}{\ensuremath{\mathrm{nat}}}%  % 
\newcommand{\comp}{\ensuremath{\mathsf{c}}}%
\newcommand{\ind}{\ensuremath{\mathrm{Ind}}}%
\newcommand{\one}{\ensuremath{\mb{1}}}%

\newcommand{\fac}[1]{\ensuremath{{|\!|#1|\!|}}}	% factorization //
\newcommand{\sph}{\ensuremath{\mathrm{Sph}}}	% 
\newcommand{\st}{\ensuremath{{{\rm St}}}}%

\usepackage{xspace}
\newcommand\ie{i.e.\xspace}
\newcommand\eg{e.g.\xspace}

\makeatletter
\pgfqkeys{/tikz/commutative diagrams}{
  row sep/.code={\tikzcd@sep{row}{#1}{}},
  column sep/.code={\tikzcd@sep{column}{#1}{}},
  bo row sep/.code={\tikzcd@sep{row}{#1}{between origins}},
  bo column sep/.code={\tikzcd@sep{column}{#1}{between origins}},
  bo column sep/.default=normal,
  bo row sep/.default=normal,
}
\def\tikzcd@sep#1#2#3{% re-defintion of original package macro!
  \pgfkeysifdefined{/tikz/commutative diagrams/#1 sep/#2}%
    {\pgfkeysalso{/tikz/#1 sep={\ifx\\#3\\1*\else1.7*\fi\pgfkeysvalueof{/tikz/commutative diagrams/#1 sep/#2},#3}}}%
    {\pgfkeysalso{/tikz/#1 sep={#2,#3}}}}
\makeatother

\newcommand\sat{^\mathsf{s}}
\newcommand\ac{^\mathsf{a}}
\newcommand\cg{^\mathsf{c}}

\begin{document}

\title{Left exact monoidal localizations from tidy maps}
\author{
Mathieu Anel%
\footnote{Laboratoire J.-A. Dieudonn\'e, Universit\'e C\^ote d'Azur, mathieu.anel@protonmail.com} ,
Georg Biedermann%
\footnote{SFB 1085 Higher Invariants, Universit\"at Regensburg, gbm@posteo.de} ,
Eric Finster%
\footnote{University of Birmingham, e.l.finster@bham.ac.uk} ,
and Andr\'{e} Joyal%
\footnote{CIRGET, UQ\`AM. joyal.andre@uqam.ca} 
}

\maketitle

\begin{abstract}
We put Goodwillie's calculus of functors and Weiss' orthogonal calculus in a unified framework.
We do so in two ways.
On the one hand, the relevant categories are all symmetric monoidal and controlled by their compact objects.
We introduce the notion of a tidy map as a means to generate symmetric monoidal localizations in this setting.
These localizations are always left exact.
Then we show that both the Goodwillie and Weiss towers are generated by such maps.
On the other hand, the relevant categories are also topoi, for which there is a general theory of completion towers of left exact localizations.
We had shown in previous work that the Goodwillie tower is an instance of such a tower.
We show here that the Weiss tower is a completion tower as well, and therefore that the general theory of these towers applies to orthogonal calculus.
\end{abstract}

\setcounter{tocdepth}{2}
\tableofcontents

\section{Introduction}

The purpose of this work is to put Goodwillie's calculus of functors and Weiss' orthogonal calculus in a common framework (thus answering a folkloric question, see the last paragraph of \cite{Arone-Ching:chapter}).
Both calculi produce towers of localizations of presheaf categories, and we will propose two perspectives on these towers:
\begin{itemize}[itemsep=0pt, topsep=3pt]
\item by looking at them as towers of symmetric monoidal localizations, and
\item by looking at them as towers of left exact localizations of \oo topoi, and proving that they are instances of completion towers in the sense of \cite[4.2]{ABFJ:product}.
\end{itemize}
The method to do so will be based on the following facts:
\begin{itemize}[itemsep=0pt, topsep=3pt]
\item all the localizations in these towers are symmetric monoidal localizations generated by inverting a single map having a distinguished property that we call \emph{tidyness}, and
\item localizations associated to a tidy map are always left exact.
\end{itemize}
The theory of completion towers is worked out in detail in our previous paper \cite{ABFJ:product} which the interested reader may consult for details.
The focus of this paper, therefore, is on the development of the other perspective using monoidal localizations, as well as showing that the two perspectives are appropriately compatible. In particular, our methods provide ways to produce localizations that are both monoidal and left exact (\ie monoidal flat).

We will now present our results in more details.

\subsection{Monoidal point of view}
By a symmetric monoidal presentable category we will mean a commutative monoid object in the category of presentable categories and cocontinuous functors.
Recall that an object $x$ in a presentable category $\mc C$ is called \emph{compact} if $\map(x,-):\mc C\to \mc S$ preserves filtered colimits and that $\mc C$ is called \emph{$\omega$\=/presentable} if $\mc C = \ind(\mc C_0)$ where $\mc C_0$ is the full subcategory of compact objects.

\begin{defn-intro}{Definition~\ref{Def-confined}}	
A symmetric monoidal presentable category $\mc V$ is \emph{confined} if 
it is $\omega$\=/presentable, 
its unit object is compact, 
and the tensor of two compact objects is compact. 
A symmetric monoidal and cocontinuous functor between two confined categories is called \emph{confined} if it sends compact objects to compact objects.
\end{defn-intro}

Every such $\mc V$ is in particular \emph{monoidal closed} and we will denote the internal hom by $[-,-]$.
If $\one$ is the monoidal unit, any map $z:Z\to \one$ in $\mc V$ induces 
a natural transformation $t=[z,-]:\Id\to T$ from the identity of $\mc V$ to the cotensor endofunctor $T=[Z,-]$.
The colimit of endofunctors
\[
P:=\colim\ \left(\Id \xto t T\xto{tT} T^2\xto{tT^2} T^3\xto{tT^3}\,\hdots \right)
\]
comes with a natural transformation $p:\Id\to P$. 
We will say that an object $X$ is \emph{$P$\=/closed} if $p(X):X\to P(X)$ is invertible. 
The full subcategory of $P$\=/closed objects is denoted by $\mc V ^P\subset \mc V$.
We introduce the following notion.

\begin{defn-intro}{Definition~\ref{BWlemmadef}}
We will say that the map $z:Z\to \one$ is {\it tidy}, if $Z$ is compact in $\mc V$ and if the map $P(z):P(Z)\to P(\one)$ is an isomorphism.
\end{defn-intro}

The notion of a tidy map is inspired from Weiss' ideas in~\cite{Weiss95,Weiss98} on how to construct reflectors.
The following result is a simplified version of Theorem~\ref{thm:monoidal-localization}.

\begin{thm-intro}{A}
\label{intro:thmA}
If $z:Z\to \one$ is a tidy map in a confined symmetric monoidal category \mc{V}, then $\mc V^P$ is a confined category and 
\begin{enumerate}[label={\rm (\roman*)}]
\item $P$ is idempotent and defines a localization $P':\mc V\to \mc V^P$,
\item $P'$ is left exact,
\item $P'$ is symmetric monoidal for the product on $\mc V^P$ defined by $X\otimes_P Y := P(X\otimes Y)$, and
\item $P'$ is confined.
\end{enumerate}
\end{thm-intro}

We will give several applications of Theorem~\ref{intro:thmA}, summarized in Table~\ref{table:GC}.

\begin{table}[H]
\begin{center}
\caption{Applications of Theorem~\ref{intro:thmA}}
\label[table]{table:GC}
\medskip
\renewcommand{\arraystretch}{1.6}
\begin{tabularx}{\textwidth}{
|>{\hsize=.6\hsize\linewidth=\hsize\centering\arraybackslash}X
|>{\hsize=1.1\hsize\linewidth=\hsize\centering\arraybackslash}X
|>{\hsize=1.1\hsize\linewidth=\hsize\centering\arraybackslash}X
|>{\hsize=1.1\hsize\linewidth=\hsize\centering\arraybackslash}X
|>{\hsize=1.1\hsize\linewidth=\hsize\centering\arraybackslash}X|}
\cline{2-5}
\multicolumn{1}{c|}{} 
& \multirow{2}*{Weiss}
&
\multicolumn{3}{c|}{Goodwillie} \\
\cline{3-5}
\multicolumn{1}{c|}{}
&& unpointed
& pointed-to-unpointed
& pointed\\
\cline{2-5}
\hline
$\mc V$ 
& $\Fun(\mc J,\mc S)$
& $\Fun(\Fin,\mc S)$ 
& $\Fun(\Fin_\ast,\mc S)$ 
& $\Fun_\ast(\Fin_\ast,\mc S_\ast)$
\\
\hline
Special compact object 
& $\sph := \map(\mb R^1,-)$ (unit~sphere functor)
& $\Id := \map(1,-)$ {$\quad$(can. inclusion)} 
& $\Id_\circ := \map_\ast(S^0,-)$ (forgetful functor) 
& $\Id_\ast := \map_\ast(S^0,-)$ (can. inclusion)
\\
\hline
Tensor 
& Day convolution wrt direct sum $\oplus:\mc J^2 \to \mc J$
& Day convolution wrt join~product $\star:\Fin^2 \to \Fin$ 
& Day convolution wrt smash~product $\wedge:\Fin_\ast^2\to \Fin_\ast$ 
& Day convolution wrt smash~product $\wedge:\Fin_\ast^2\to \Fin_\ast$ 
\\
\hline
Unit 
& terminal functor 1
& terminal functor 1
& $\Id_\circ$ 
& $\Id_\ast$ 
\\
\hline
Tidy maps ($n\geq0$)
& $(\sph\to 1)^{\star n+1}$
& $(\Id\to 1)^{\star n+1}$
& $(1\to \Id_\circ)^{\star n+1}$ 
& $(1\to \Id_\ast)^{\star n+1}$
\\
\hline
$P$\=/closed objects 
& $n$\=/polynomial functors $\mc J\to \mc S$
& $n$\=/excisive functors $\Fin\to \mc S$ 
& $n$\=/excisive functors $\Fin_\ast\to \mc S$ 
& $n$\=/excisive functors $\Fin_\ast\to_\ast \mc S_\ast$ 
\\
\hline
\end{tabularx}
\end{center}  
\end{table}

In Table~\ref{table:GC}, $\mc S$ (resp. $\mc S _\ast$) is the \oo category of (pointed) spaces,
$\Fin\subset\mc S$ ($\Fin_\ast\subset\mc S_\ast$) are the subcategories of finite (pointed) spaces, $\mc J$ is the \oo category derived from the topological category of finite dimensional Euclidean vector spaces and linear isometries between them (see Section~\ref{sec:orthogonal-calculus}).
If $\mc C$ is a category with finite limits and finite colimits, the \emph{join product} $A\star B$ of two objects $A$ and $B$ is the pushout of the span $A\leftarrow A\times B\to B$.
When finite colimits are universal in $\mc C$, this defines a (non closed) symmetric monoidal structure on $\mc C$ whose unit is the initial object.
If $A\to C$ and $B\to C$ are two maps in $\mc C$, their \emph{fiberwise join} $(A\to C)\star (B\to C)$ is defined as the join product in the slice category $\mc C/C$ (see Subsection~\ref{sec:fiberwise-join}).
The fiberwise join powers of a map $A\to C$ are denoted $(A\to C)^{\star n}$.

\medskip
Table~\ref{table:GC} shows that Weiss and Goodwillie towers are generated in exactly the same way:
\begin{enumerate}[label=\roman*), itemsep=0pt, topsep=3pt]
\item a certain primitive tidy map is considered (case $n=0$)
\item such that all fiberwise join powers are also tidy (a fact we prove by means of connectivity estimates in Lemma~\ref{r-star-is-On-hence-Pn-equivalence137}, Lemma~\ref{tidy-pointed} and Proposition~\ref{prop:Pn-zetan-WeissA}), and
\item the stages of the tower are the monoidal localizations generated by these tidy maps.
\end{enumerate}
(We do not know if the join powers of a tidy map are always tidy maps, but finding conditions for when this is true is an interesting problem.)

As a consequence, each of the categories of local objects (last row of Table~\ref{table:GC}) inherits a symmetric monoidal structure.
In the pointed Godwillie example, when $n=1$, this recovers the construction of the smash product of spectra by Lydakis~\cite{Lydakis}.
In the Weiss calculus example, these monoidal structures are considered by Hendrian~\cite{hendrian:monoidal-orth-calc}.

Finally, we prove a version of Theorem~\ref{intro:thmA} for modules over the symmetric monoidal category $\mc V$ in Theorem~\ref{thm:module-localization}. This is applied in Theorem~\ref{thm:Goodwillie-all} to derive the Goodwillie localizations $P_n$ on the category $\Fun(\mc{C},\mc{S})$ as a module over $\Fun(\Fin,\mc{S})$ where \mc{C} is a small finitely cocomplete category with a terminal object. 
This also allows us to prove that the image of a tidy map by a confined functor is tidy in Theorem~\ref{thm:imagetidy}, a fact that can be seen to subsume the comparison between the Goodwillie towers in the unpointed and pointed settings (Corollary~\ref{cor:transport-tidy}).

\subsection{Topos point of view}

Our second unification of Goodwillie and Weiss calculi uses the notion of a completion tower of a left exact localization of \oo topoi \cite[4.2]{ABFJ:product}.
Recall from op.~cit. that we call a {\em congruence} the class of maps in an \oo topos inverted by a left exact localization, and that we define a product on congruences.
If $\mc C$ is a congruence in an \oo topos $\mc E$, the powers of $\mc C$ define a decreasing sequence $\dots\subset \mc C^2\subset \mc C$ of congruences and the corresponding tower of localizations $\mc E \to \dots \to \mc E[(\mc C^2)^{-1}]\to \mc E[\mc C^{-1}]$ is called the {\em completion tower} of $\mc C$.

We have already established in \cite[4.3]{ABFJ:product} that the 
Goodwillie towers are completion towers.
Our second main result in this paper is to show that the Weiss tower is a completion tower as well.
Let $\mc C_n$ be the congruence associated to the left exact localization generated by the tidy map $\sph^{\star n+1}\to 1$ of (see first column of Table~\ref{table:GC}).
The following is essentially Theorem~\ref{thm:ortho-is-compl}.

\begin{thm-intro}{B}
The Weiss tower is the completion tower of $\mc C_0$.
In other terms, $\mc C_n = (\mc C_0)^{n+1}$.
\end{thm-intro}

It follows immediately that the Weiss tower has all the structural features of general completion towers.
For example, this recover that its layers are stable $\infty$\=/categories \cite[Theorem~4.2.11]{ABFJ:product}, and it satisfies a generalized Blakers-Massey theorem (see Subsection~\ref{subsec:blakers-massey}).
  
\paragraph{Acknowledgments:} 
We thank Rune Haugseng and Hadrian Heine for sharing with us their expertise on enriched higher category theory. 
We thank Sil Linskens and Marco Giustetto (from whom we took the word ``docile'') for their careful reading an earlier version of the paper.

Mathieu Anel has received funding from the European Research Council (ERC) under the European Union's Ninth Framework Programme Horizon Europe (ERC Synergy Project Malinca, Grant Agreement n. 101167526).
Georg Biedermann would like to acknowledge the support of the SFB 1085 Higher Invariants at the Universit\"at Regensburg and in particular the help and encouragement of Denis-Charles Cisinski.
Eric Finster acknowledges the support of the U.S. Air Force Office of Special Research award FA9550-23-1-0029.

\section{Tools from category theory}
\label{sec:tools}

\subsection{Enriched higher category theory}
\label{subsec:enriched-higher-cats}

Symmetric monoidal structures for ${(\infty,1)}$\=/category theory are developed in 
Lurie~\cite{HA}, and with an alternative approach in Gepner--Haugseng~\cite{GH19}. The theory of modules over symmetric monoidal ${(\infty,1)}$\=/categories (and more general notions) are developped in the work of Heine, see in particular~\cite{Heine:equiv-enriched-inf-cat-inf-cat-weak-action,Heine:bi-enriched}.
In this article all categories are ${(\infty,1)}$\=/categories unless we explicitly say otherwise.

Let \mc{V} be a symmetric monoidal closed category.
If $\mc{M}$ is a \mc{V}\=/category, we write
\[
[-,-]_{\mc{V}}\ \text{ or  simply }\ [-,-]:\mc{M}^{\op}\times\mc{M}\to\mc{V}
\]
for the \mc{V}\=/enrichment of \mc{M}. 
A {\it closed \mc{V}\=/module category} is a \mc{V}\=/enriched category that is tensored and cotensored over \mc{V}.
In this case we write (again) 
\[
-\otimes-: \mc{V}\times\mc{M}\to\mc{M}
\]
for the tensor and denote the cotensor by
\[
\{-,-\}: \mc{V}^{\op}\times\mc{M}\to\mc{M}\,.
\]
Then we have natural isomorphisms
\[
[A\otimes M,N]_{\mc{V}}\cong [A,[M,N]_{\mc{V}}]\cong [M,\{A,N\}]_{\mc{V}}\,,
\]
for $A$ in \mc{V} and $M,N$ in $\mc{M}$.
Obviously, \mc V is a module over itself.

\begin{lemma}\label{fact:1}
Let $A\to B$ be a morphism in \mc{V}. Let \mc{M} be a closed \mc{V}\=/module. Then the induced natural transformation $\{B,-\}\to \{A,-\}$ 
is a \mc{V}\=/enriched natural transformation of \mc{V}\=/enriched functors.
\end{lemma}

\begin{proof}
It is shown in~\cite[Proposition 2.116]{Heine:bi-enriched} that there is an equivalence of categories between enriched left adjoints and enriched right adjoints given by mapping them to each other.
Now one observes that tensoring with a fixed object $A\otimes-$ is (lax) monoidal, and hence \mc{V}\=/enriched. Therefore the right adjoint cotensor $\{A,-\}$ 
is \mc{V}\=/enriched and the map $A\to B$, that induces a \mc{V}\=/enriched natural transformation of tensors, in turn induces a \mc{V}\=/enriched natural transformations of cotensors.  
\end{proof}

Let $\Fun_{\mc{V}}(\mc{C},\mc{D})$ denote the category of \mc{V}\=/enriched functors between \mc{V}\=/closed modules with morphisms given by \mc{V}\=/enriched natural transformations.

\begin{proposition}\label{prop:Heine-rules} 
Let \mc{V} be a symmetric monoidal closed category and \mc{C} and \mc{D} two closed \mc{V}\=/modules. Let $u:\Fun_{\mc{V}}(\mc{C},\mc{D})\to \Fun(\mc{C},\mc{D})$ be the functor forgetting the \mc{V}\=/enrichment.
\begin{enumerate}[label={\rm (\roman*)}]
  \item\label{prop:Heine-rules:1}
The functor $u$ is conservative. 
  \item\label{prop:Heine-rules:2}
If \mc{D} is complete, then the category $\Fun_{\mc{V}}(\mc{C},\mc{D})$ is complete and the functor $u$ is continuous. 
  \item\label{prop:Heine-rules:3}
If \mc{D} is cocomplete, then the category $\Fun_{\mc{V}}(\mc{C},\mc{D})$ is cocomplete and the functor $\Fun_{\mc{V}}(\mc{C},\mc{D})\to \Fun(\mc{C},\mc{D})$ forgetting the enrichment is cocontinuous.
\end{enumerate} 
Under all these assumptions, it follows that limits and colimits exist  in $\Fun_{\mc{V}}(\mc{C},\mc{D})$ and are computed objectwise.
\end{proposition}

\begin{proof}
Heine's results in~\cite{Heine:bi-enriched} are proved in a more general setting of (weakly) left, right and bienriched ${(\infty,1)}$\=/categories over two non-symmetric $\infty$\=/operads \mc{V} and \mc{W}. It is shown that ${(\infty,1)}$\=/categories left enriched in any two non-symmetric $\infty$\=/operads \mc{V} and \mc{W} form an $(\infty,2)$\=/category, which is denoted by $\mbox{}_\mc{V}{\rm LEnr}_{\mc{W}}$ by~\cite[Remark 2.136]{Heine:bi-enriched}.
We can specialize to our setting: if \mc{V} is a symmetric monoidal ${(\infty,1)}$\=/category and \mc{W} is the initial $\infty$\=/operad, whose space of colors is empty and which is denoted by $\emptyset$, then left $(\mc{V},\emptyset)$\=/enriched ${(\infty,1)}$\=/categories are precisely \mc{V}\=/enriched ${(\infty,1)}$\=/categories by~\cite[Remark 2.140]{Heine:bi-enriched} and the $(\infty,2)$\=/category $\mbox{}_\mc{V}{\rm LEnr}_{\emptyset}$ is the $(\infty,2)$\=/category of \mc{V}\=/enriched ${(\infty,1)}$\=/categories by the same remark. 

In~\cite[Remark 2.136]{Heine:bi-enriched} the functor $U: \mbox{}_{\mc{V}}{\rm LEnr}_{\mc{W}} \to {\rm Cat}_{(\infty,1)}$ forgetting the enrichment is constructed as a functor of $(\infty,2)$\=/categories. 
This forgetful functor $U$ is locally conservative, \ie induces on morphism-${(\infty,1)}$\=/categories conservative functors by~\cite[Corollary 2.138]{Heine:bi-enriched}. Therefore our functor $u$ is conservative, proving~\ref{prop:Heine-rules:1}.

By \cite[Theorem 4.43\,(2)]{Heine:bi-enriched} applied to enrichment, the forgetful functor $u$ is monadic if \mc{D} is (left) tensored over \mc{V} and has colimits preserved by the left action. The latter is covered since in the closed \mc{V}\=/module \mc{D} the tensor has a right adjoint.
This implies~\ref{prop:Heine-rules:2}. Since limits in $\Fun(\mc{C},\mc{D})$ are computed objectwise, the same now holds for the category $\Fun_{\mc{V}}(\mc{C},\mc{D})$.

For~\ref{prop:Heine-rules:3} we can specialize~\cite[Lemma 3.74\,(1)]{Heine:equiv-enriched-inf-cat-inf-cat-weak-action} and conclude that the category $\Fun_{\mc{V}}(\mc{C},\mc{D})$ is cocomplete and for every object $X$ in \mc{C} the evaluation functor
\begin{align*}
\ev_X : \Fun_{\mc{V}}(\mc{C},\mc{D}) &\longrightarrow\mc{D} \\
F&\longmapsto \ev_X(F)=F(X)	
\end{align*}
preserves colimits. We deduce that the functor $u$ is cocontinuous and that colimits in $\Fun_{\mc{V}}(\mc{C},\mc{D})$ are computed objectwise. 
\end{proof}

We would like to stress the importance of the enriched theory in the higher categorical context for our work since it is essential in the proof of Proposition~\ref{prop:Weiss-trick}. 

\subsection{\texorpdfstring{$\omega$\=/Presentable}{Omega-presentable} categories}
\label{subsec:compact}

Let $\mc{E}$ be a category with filtered colimits. An object $K$ in \mc{E} is called {\it compact}, or more accurately {\it $\omega$\=/compact}, if the functor $\map(K,-):\mc{E}\to \mc{S}$ preserves filtered colimits. Here \mc{S} is the category of spaces, aka $\infty$\=/groupoids.
We denote the full subcategory of compact objects of $\mc{E}$ by $\comp(\mc{E})$.
This subcategory is closed under finite colimits and retracts~\cite[Corollary 5.3.4.15, Remark 5.3.4.16]{HTT}. 

\begin{definition}[{\cite[Definition 20.4.1.7]{Lurie:sag}}] \label{dense-subcategory}
A small full subcategory $\mc{D}$ of a category $\mc{E}$ is said to be {\it dense} if the identity functor of \mc{E} can be given as a left Kan extension of the inclusion $\mc{D}\subset\mc{E}$.
\end{definition} 

\begin{lemma}\label{lem:nerve-colimit}
Let \mc{D} be a small full subcategory of \mc{E}. For an object $E$ in \mc{E} let $\mc{D}/E$ denote the slice category.
The following are equivalent:
\begin{enumerate}[label={\rm (\roman*)}]
  \item
The subcategory \mc{D} is dense in \mc{E}.
  \item 
For every object $E$ in \mc{E} we have $E\cong\colim\limits_{\mc{D}/E} D$.
  \item
The restricted Yoneda functor $\mc{E}\to [\mc{D}^{\op}, \mc{S}]$ is fully faithful. 
\end{enumerate}
\end{lemma}

\begin{proof}
This is proved in \cite[Remarks 20.4.1.2 and 20.4.1.5]{Lurie:sag}.
\end{proof}

\begin{lemma} \label{lem:detectinvertible}
Let $\mc{D}\subset \mc{E}$ be a small dense full subcategory of a category $\mc{E}$ with filtered colimits.
Then a morphism $f:X\to Y$ in \mc{E} is invertible if and only if $\map(D,f):\map(D,X)\to \map(D,Y)$ is invertible for every object $D$ in $\mc{D}$.
\end{lemma}

\begin{proof} 
The (restricted) Yoneda functor  $ \mc{E}\to \Fun(\mc{D}^{\op},\mc{S})$
is fully faithful, since the subcategory $\mc{D}$ is dense.
But any fully faithful functor is conservative.
\end{proof}

\begin{definition} \label{omega-pres}
A category $\mc{E}$ is said to be {\it $\omega$\=/presentable} if it is cocomplete and its full subcategory of compact objects $\comp(\mc{E})\subset \mc{E}$ is small and dense in $\mc{E}$.
\end{definition}

An $\omega$\=/presentable category is complete and cocomplete.
\begin{example}
Let $\mc{D}$ be a small category and consider the category $\Fun(\mc{D},\mc{S})$. In it every representable functor is compact and any functor is a colimit of representable ones. Thus the category $\Fun(\mc{D},\mc{S})$ is $\omega$\=/presentable. 
We will say that a functor is {\it finitely presentable} if  it is a finite colimit of representable functors. It is compact if and only if it is a retract of a finitely presentable functor.
\end{example}

\begin{lemma} \label{lem:objects-are-filtered-colimits}
If the category \mc{E} is $\omega$\=/presentable, then every object in \mc{E} is a filtered colimit of compact objects.
\end{lemma}

\begin{proof}
Since \mc{E} is $\omega$\=/presentable, for an arbitrary object $E$ in \mc{E}, we have $E\cong\colim_{\comp(\mc{E})/E} D$ by Lemma~\ref{lem:nerve-colimit}. 
Compact objects are closed under finite colimits and colimits in $\comp(\mc{E})/E$ are computed in $\comp(\mc{E})$.
Thus, the slice category $\comp(\mc{E})/E$ is filtered.
\end{proof}

Consider a natural transformation $f:A\to B$ between two increasing sequences 
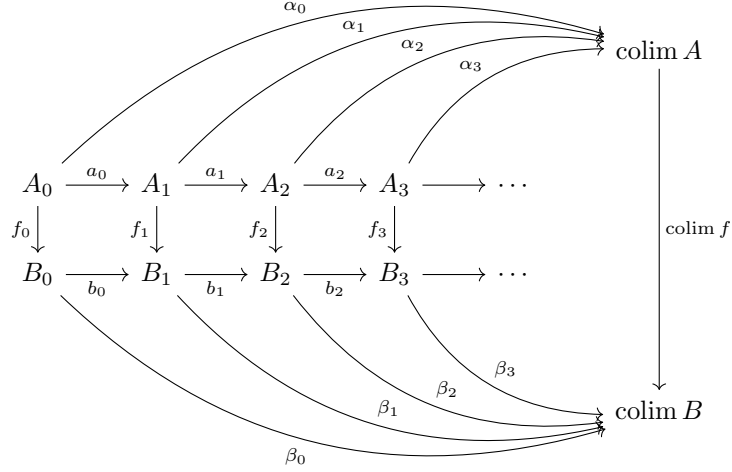
\begin{equation*}
\begin{tikzcd} 
&&&&&  \colim A\ar[ddddd,"{\colim f}"]  \\
&&&&&  \\
A_0 \ar[d,"{f_0}"'] \ar[r,"{a_0}"] \ar[uurrrrr, "\alpha_0", bend left] 
& A_1\ar[r,"{a_1}"] \ar[d,"{f_1}"'] \ar[uurrrr, "\alpha_1", bend left] 
& A_2 \ar[d,"{f_2}"'] \ar[r,"{a_2}"]  \ar[uurrr, "\alpha_2", bend left]  
& A_3 \ar[d,"{f_3}"'] \ar[uurr, "\alpha_3", bend left] \ar[r]  & \cdots & 
 \\
B_0 \ar[ddrrrrr, "\beta_0"', bend right]  \ar[r,"{b_0}"'] 
& B_1  \ar[r,"{b_1}"'] \ar[ddrrrr, "\beta_1", bend right]  
& B_2\ar[r,"{b_2}"']  \ar[ddrrr, "\beta_2", bend right]
& B_3 \ar[ddrr, "\beta_3", bend right] \ar[r] & \cdots &  \\
&&&&&  \\
&&&&&  \colim B
\end{tikzcd}
\end{equation*}
in \mc{E} and write $\alpha_n:A_n\to \colim A$ and $\beta_n:B_n\to \colim B$ for the conical maps.

\begin{lemma} \label{Whiteheadcolimit17}
Let $\mc{E}$ be an $\omega$\=/presentable category and let $f:A\to B$ be a natural transformation between two sequences in \mc{E}. If the square 
\begin{equation*}
\begin{tikzcd}
A_n \ar[d,"f_n"'] \ar[r,"{\alpha_n}"] & \colim(A)
\ar[d,"{\colim(f)}"]
 \\
B_n\ar[r,"{\beta_n}"] & \colim(B)
\end{tikzcd}
\end{equation*}
has a diagonal filler for every $n\geq 0$, then the map $\colim(f)$ is invertible.
\end{lemma}

\begin{proof} 
Consider first the case $\mc{E}=\mc{S}$. Here the classical Whitehead Theorem shows, that the map $\colim(f)$ is invertible if and only if every square
\begin{equation*}
\begin{tikzcd}
S^n \ar[d] \ar[r,"{x}"] &  \colim(A) 
\ar[d,"{\colim(f)}"]
 \\
1 \ar[r,"{y}"] &  \colim(B)
\end{tikzcd}
\end{equation*}
has a diagonal filler. 
But this square  can be factored as follows:
\begin{equation*}
\begin{tikzcd}
S^n \ar[d] \ar[r,"{x'}"] &  A_k \ar[d,"f_k"'] \ar[r,"{\alpha_k}"] & \colim(A) \ar[d,"{\colim(f)}"] \\
1 \ar[r,"{y'}"]  & B_k\ar[r,"{\beta_k}"] & \colim(B)\,,
\end{tikzcd}
\end{equation*}
for some $k\geq 0$, since $S^n\to 1$ is a map between compact spaces.
The composite square has a diagonal filler since the right hand square has a diagonal filler by the hypothesis. This proves the lemma in the the case where $\mc{E}=\mc{S}$. 

Let us now return to the general case of an $\omega$\=/presentable category $\mc{E}$. By Lemma~\ref{lem:detectinvertible} it suffices to show that $\map(K,\colim(f))$ is invertible for every compact object $K$ in \mc{E}.
But the functor $\map(K,-):\mc{E}\to \mc{S}$ preserves filtered colimits, since $K$ is compact. Hence $\map(K,\colim(f))$ is the colimit of the natural transformation $\map(K,f): \map(K,A)\to \map(K, B)$,
\begin{equation*} 
\begin{tikzcd}
\map(K,A_n) \ar[d,"{\map(K,f_n)}"'] \ar[rr,"{\map(K,\alpha_n)}"] && \map(K,\colim(A))
\ar[d,"{\map(K,\colim(f))}"]
 \\
\map(K,B_n)\ar[rr,"{\map(K,\beta_n)}"] &&\map(K,\colim(B))\,.
\end{tikzcd}
\end{equation*}
For every $n\geq 0$ this square obtains an induced diagonal filler from the square above. Hence the map $\map(K,\colim(f))$ is invertible by the first part of the proof. It follows that $\colim(f)$ is invertible.
\end{proof}

\begin{lemma}\label{lem:filt-colim-lex}
In an $\omega$\=/presentable category filtered colimits are left exact.
\end{lemma}

\begin{proof}
Let \mc{F} be a filtered category and \mc{D} be a finite category. The claim is that for any functor $G:\mc{F}\times\mc{D}\to\mc{V}$ the canonical map
\[
A:=\colim_{\mc{F}}\lim_{\mc{D}} G(f,d) \xto{\sigma} \lim_{\mc{D}}\colim_{\mc{F}} G(f,d)=:B
\]
is an isomorphism. The result is proven in \cite[Proposition 5.3.3.3]{HTT} for $\mc{V}=\mc{S}$. The claim follows from Lemma~\ref{lem:detectinvertible} once we show that $\map(K,\sigma):\map(K,A)\to\map(K,B)$ is an isomorphism for arbitrary compact $K$ in \mc{V}. 
But both sides turn out to be canonically isomorphic to the space $\colim_{\mc{F}}\lim_{\mc{D}}\map(K,G(f,d))$.
\end{proof}

\section{Confined categories}
\label{sec:confined-cats}

\subsection{Confined functors}

\begin{definition} \label{def:confinedfunctor}
A functor betwen $\omega$\=/presentable categories is {\it confined} if it is cocontinuous and takes compact objects to compact objects.
\end{definition}

\begin{remark}
\label{rem:confined}
The category of $\omega$\=/presentable categories and confined functors is equivalent to that of idempotent complete and finitely cocomplete categories and functors preserving finite colimits. 
The equivalence is given by extracting the compact objects in one direction and by Ind-completion in the other.
\end{remark}

Recall from~\cite[Corollary 5.5.2.9]{HTT} that every cocontinuous functor between presentable categories  $\phi:\mc{E}\to \mc{F}$ has a right adjoint $\phi_\star:\mc{F}\to \mc{E}$.

\begin{proposition} \label{prop:confinedbasic2} 
Let $L:\mc{E}\to\mc{F}$ be a cocontinuous functor between presentable categories and let $R:\mc{F}\to\mc{E}$ its right adjoint. 
Let $\mc{D}\subset\comp(\mc{E})$ be a small full subcategory of the compact objects of \mc{E} that is dense in \mc{E}. 
Then the following are equivalent:
\begin{enumerate}[label={\rm (\roman*)}]
   \item\label{prop:confinedbasic2:1} The functor $L$ is confined.
   \item\label{prop:confinedbasic2:2} For every $D$ in $\mc{D}$ the object $L(D)$ is compact. 
   \item\label{prop:confinedbasic2:3} The functor $R$ preserves filtered colimits.
\end{enumerate}
\end{proposition}

\begin{proof} 
Obviously~\ref{prop:confinedbasic2:1} implies~\ref{prop:confinedbasic2:2}. 
Now assume~\ref{prop:confinedbasic2:2} and let us show that $R$ preserves filtered colimits. 
By Lemma~\ref{lem:detectinvertible} it suffices to show that $\map(D,R(-))$ does so for all $D$ in \mc{D}, since \mc{D} is dense in \mc{E}. But this is true since $\map(D,R(-))\cong\map(L(D),-)$ and $L(D)$ is compact by assumption. So~\ref{prop:confinedbasic2:2} implies~\ref{prop:confinedbasic2:3}.
Suppose that $R$ preserves filtered colimits and let $K$ in \mc{E} be compact. Then the functor $\map (LK,-)\cong\map(K,R(-)):\mc{F}\to \mc{S}$ preserves filtered colimits. Hence $LK$ is compact. So~\ref{prop:confinedbasic2:3} implies~\ref{prop:confinedbasic2:1}.
\end{proof} 

Let $\mc{E}$ be a presentable category.
If $\mc{A}$ is a small category, then every functor $\phi:\mc{A}^{\op}\to \mc{E}$ has a left Kan extension $\phi_!:\Fun(\mc{A},\mc{S})\to  \mc{E}$ along the Yoneda functor $y:\mc{A}^{\op}\to\Fun(\mc{A},\mc{S})$.

\begin{corollary}\label{cor:shriek-confined}
If $\phi(\mc{A}^{\op})\subset \comp(\mc{E})$, then the functor $\phi_!:\Fun(\mc{A},\mc{S})\to  \mc{E}$ is confined.
\end{corollary}

\begin{proof} 
The subcategory of representable functors is dense and every object in it is compact.
Moreover the functor $\phi_!$ takes every representable functor to a compact object in \mc{E}, since $\phi_!(y(a))\cong \phi(a)$ is compact for every $a$ in $\mc{A}$, $y$ the Yoneda functor.
Then Proposition~\ref{prop:confinedbasic2} implies that $\phi_!$ is confined.
\end{proof}

\begin{example}\label{exam:lKan-confined}
If $\phi:\mc{A}\to \mc{B}$ is functor between small categories, then the functor $\phi^*: \Fun( \mc{B},\mc{S})\to \Fun( \mc{A},\mc{S})$ has a left adjoint $\phi_!: \Fun( \mc{A},\mc{S})\to \Fun( \mc{B}, \mc{S})$ which is confined.
\end{example}

\subsection{Confined symmetric monoidal categories}
\label{subsec:csmc-cat}

\begin{definition} \label{Def-confined} 
We will say that a symmetric monoidal category is {\it confined} if 
\begin{enumerate}[label={\rm (\roman*)}]
\item it is $\omega$\=/presentable, 
\item the unit object is compact, and
\item the monoidal product of two compact objects is compact.
\end{enumerate}
\end{definition}

If $\mc C$ is a confined monoidal category, the monoidal structure interacts well with the $\omega$\=/presentability.
In particular, the assumption that the monoidal product of two compact objects is compact is useful to derive Lemma~\ref{lem:compact-cotensors-good} showing that the internal hom from a compact object preserves filtered colimits.

Morphisms between confined symmetric monoidal categories are taken to be the symmetric monoidal functors which are confined in the sense of Definition~\ref{def:confinedfunctor}.
Note that confined symmetric monoidal categories are automatically closed.
However these morphisms need not preserve the internal hom.

\begin{remark}
\label{rem:confined-monoidal}
The category of confined symmetric monoidal categories and confined symmetric monoi\-dal functors is equivalent to the category of symmetric monoids in $\omega$\=/presentable categories and confined functors.
Recall from Remark~\ref{rem:confined} that the latter category is equivalent to idempotent complete and finitely cocomplete categories and functors preserving finite colimits. 
Therefore, the category of confined symmetric monoidal categories is equivalent to the category of symmetric monoids in idempotent complete and finitely cocomplete categories.
The equivalence is given by extracting the compact objects in one direction and by Ind-completion in the other.
\end{remark}

\begin{lemma} \label{sufficientforconfinedtensor}
Let \mc{V} be a symmetric monoidal category.
Suppose that \mc{V} is $\omega$\=/presentable, and that $\mc{D}\subset \mc{V}$ is a dense subcategory of compact objects of \mc{V}.
Assume also that $\one$ is compact and that for all  $X,Y$ in $\mc{D}$ the tensor $X\otimes Y$ is compact. Then \mc{V} is confined.
\end{lemma}

\begin{proof} 
For any object $A$ the functor $A\otimes -$ is cocontinuous,  since it has a right adjoint $[A,-]$.
If $A$ is in $\mc{D}$, then $A\otimes \mc{D}\subset \comp(\mc{V})$, since $\mc{D} \otimes  \mc{D}  \subset \comp(\mc{V})$ by hypo\-the\-sis.
Now it follows from Proposition~\ref{prop:confinedbasic2} that the functor $A\otimes -: \mc{V}\to \mc{V}$ is confined for every $A$ in \mc{D}. We need to extend this to any $A$ in $\comp(\mc{V})$.

We have just proved that $\mc{D}\otimes B \subset \comp(\mc{V})$ for every object $B$ in $\comp(\mc{V})$. Hence the functor $-\otimes B: \mc{V}\to \mc{V}$ is confined for every object $B$ in $\comp(\mc{V})$ by Proposition~\ref{prop:confinedbasic2}.
Thus, $\comp(\mc{V})\otimes B\subset \comp(\mc{V})$ for every object $B$ in $\comp(\mc{V})$.
Thus, $\comp(\mc{V})\otimes \comp(\mc{V})\subset \comp(\mc{V})$ and \mc{V} is a confined symmetric monoidal category.
\end{proof}

\begin{examples}\label{exam:confined-sym-mon-cats}
A short list of examples follows.
\begin{enumerate}[label={\rm (\roman*)}]
  \item\label{exam:confined-sym-mon-cats:1}
The category of spaces equipped with the cartesian monoidal structure is confined.
  \item\label{exam:confined-sym-mon-cats:2}
The category of pointed spaces with the smash product is confined.
  \item\label{exam:confined-sym-mon-cats:3}
The category of spectra equipped with the smash product is confined.
   \item\label{exam:confined-sym-mon-cats:4}
The category of small categories is cartesian closed symmetric monoidal and confined.
  \item\label{exam:confined-sym-mon-cats:5}
The category of simplicial spaces is cartesian closed and confined. 
The subcategory of representable functors satisfies the conditions of Lemma~\ref{sufficientforconfinedtensor}.
   \item\label{exam:confined-sym-mon-cats:6}
If $\mb{V}$ is a small symmetric monoidal category, then the functor category $\Fun(\mb{V},\mc{S})$ equipped with the Day convolution product is a ${(\infty,1)}$\=/sym\-met\-ric monoidal closed category, see \cite[Example 2.2.6.17,  Corollary 4.8.1.12, Remark 4.8.1.13]{HA} or~\cite{Glasman:day-conv}.
It is confined by Lemma~\ref{sufficientforconfinedtensor} using the representable functors as the small dense subcategory \mc{D}.
\end{enumerate}
\end{examples}

\subsection{Confined \texorpdfstring{\mc{V}\=/}{V-}modules}
\label{subsec:confined-modules}

\begin{definition} 
Let \mc{M} be a closed \mc{V}\=/module over a confined symmetric monoidal category  \mc{V}. We will say that \mc{M} is a {\it confined \mc{V}\=/module} if the category \mc{M} is $\omega$\=/presentable and the tensor product $A\otimes M$ of a compact object $M$ in \mc{M} with a compact object $A$ in $\mc{V}$ is compact in \mc M.
\end{definition}

\begin{example}\label{exam:module-over-itself}
Every confined symmetric monoidal category \mc{V} is confined as a closed \mc{V}\=/module over itself. 
\end{example}

\begin{example}\label{exam:conf-sym-mon-fun-yields-conf-mod}
If $\phi:\mc{V}\to \mc{E}$ is a cocontinuous symmetric monoidal functor between presentable symmetric monoidal categories, then the category \mc{E} has the structure of a closed \mc{V}\=/module by defining
\begin{equation*} 
  [E,F]_{\mc{V}}:=\phi_{*} [E,F]_{\mc{E}},\quad  A\otimes E:=\phi(A)\otimes E \quad {\rm and} \quad \{A,F\}:=[\phi(A),F]_{\mc{E}}\,,
\end{equation*}
for $E,F$ in \mc{E} and $A$ in \mc{V}. Here $\phi_{*}$ is the right adjoint to $\phi$.
If $\phi:\mc{V}\to \mc{E}$ is a confined symmetric monoidal functor between confined symmetric monoidal categories, then \mc{E} becomes a confined \mc{V}\=/module.
\end{example}

\begin{example} \label{ex:sym-mon-conf}
Let $\sigma: \mb{V}\to \mb{E}$ be a symmetric monoidal functor between small symmetric monoidal categories.
Then the functor $\sigma^*:\Fun(\mb{E},\mc{S})\to \Fun(\mb{V},\mc{S})$ has a left adjoint $\sigma_!:\Fun(\mb{V},\mc{S})\to \Fun(\mb{E},\mc{S})$ obtained by left Kan extension along $\sigma$. With respect to the Day convolution products on both functor categories, $\sigma_!$ is symmetric monoidal and confined by Example~\ref{exam:lKan-confined}. By Example~\ref{exam:conf-sym-mon-fun-yields-conf-mod} the category $\Fun(\mb{E},\mc{S})$ becomes a confined $\Fun(\mb{V},\mc{S})$\=/module.
\end{example}

\begin{example}\label{ex:Day-action-from-action} 
Let  $\mathbbm{V}$ be a small symmetric monoidal category and $\mathbbm{M}$ be a category equipped with an action 
\[
\oplus: \mathbbm{V}\times \mathbbm{M}\to \mathbbm{M}
\]
which is coherently associative and unital.
Then the functor category $\mc{M}=\Fun(\mb{M},\mc{S})$ has the structure of a closed module over the symmetric monoidal closed category $\mc{V}= \Fun(\mb{V},\mc{S})$. 
For $F$ in \mc{V} and $M$ in \mc{M} the {\it tensor product} $F\otimes M: \mathbbm{M} \to \mc{S}$ is calculated by the formula
\[
(F\otimes M)(n)=\int^{a\in \mathbbm{V},m\in \mathbbm{M}} F(a)\times M(m)\times \map(a\oplus m,n)\,,
\]
for every $n$ in $\mathbbm{M}$. If write $R^x=\map(x,-)$ for covariant representable functors, then $R^a\otimes R^m=R^{a\oplus m}$ for every $a$ in $\mathbbm{V}$ and $m$ in $\mathbbm{V}$.
The enrichment $[M,N]_{\mc{V}}$ between $M$ and $N$ in $\mc{M}$ is calculated by the formula
\[
[M,N]_{\mc{V}}(a)=\nat(M, N(a\oplus -))\,,
\]
for every $a$ in $\mathbbm{V}$.
In particular, $[R^m, N]_{\mc{V}}=N(-\oplus m):\mathbbm{V}\to \mc{S}$. 
The {\it cotensor product}  $\{F,M\}: \mathbbm{M} \to \mc{S}$ is calculated by the formula 
\[
\{F,M\}(m)=\nat(F,M(-\oplus m))\,,
\]
for every $m$ in $\mathbbm{M}$. In particular, $\{R^a,M\}=M(a\oplus -): \mc{M}\to  \mc{S}$ for every $a$ in $\mathbbm{V}$. 
The formula $R^a\otimes R^m=R^{a\oplus m}$ shows that the \mc{V}\=/module \mc{M} is confined.
\end{example}

\subsection{Docile functors}

\begin{definition}\label{def:preserve-cotensor}
A \mc{V}\=/functor $F:\mc{V}\to \mc{V}$ preserves {\it compact cotensors} if the coassembly map $\gamma(Z,X):F[Z,X]\to [Z, FX]$ is invertible for all $X$ and every compact $Z$.
\end{definition}

Compact cotensors can be regarded as a \mc{V}\=/enriched version of finite limits and should therefore be taken into account in an enriched version of left exactness.
\begin{definition} 
We say that a \mc{V}\=/functor is
\begin{enumerate}[label={\rm (\roman*)}]
   \item
\mc{V}\=/{\it left exact} if it preserves finite limits and compact cotensors. 
   \item
{\it docile} if it is \mc{V}\=/left exact and preserves filtered colimits.
\end{enumerate}
\end{definition}

The reason for introducing confined categories is to prove the following.
\begin{lemma}\label{lem:compact-cotensors-good}
Let \mc{V} be a confined symmetric monoidal category and \mc{M} a confined \mc{V}\=/module. If $C$ is a compact object of \mc{V}, then the functors $[C,-]:\mc{V}\to\mc{V}$ and $\{C,-\}:\mc{M}\to\mc{M}$ are docile.
\end{lemma}

\begin{proof}
The functors $[C,-]$ and $\{C,-\}$ are \mc{V}\=/enriched by Lemma~\ref{fact:1} and preserve all limits, not just finite ones. They also preserves all cotensors , not just compact ones: for all $B$ in \mc{V} the coassembly map $\gamma(B,X):\{B, \{C,X\}\} \to \{C,\{B,X\}\}$ is composed of the natural isomorphisms
\[
\gamma(B,X): \{B, \{C,X\}\} \xto{\cong}  \{C\otimes B, X\} \xto{\cong} \{C,\{B,X\}\}\,.
\]
Hence the map $\gamma(B,X)$ is invertible for all objects $B,C$ in \mc{V} and $X$ in \mc{M}. And similarly for $[C,-]$.  

Since \mc{V} and \mc{M} are assumed to be confined, the tensor $C\otimes -$ is a confined functor, both as endofunctor of \mc{V} and of \mc{M}. So by Proposition~\ref{prop:confinedbasic2} the respective right adjoints $[C,-]$ and $\{C,-\}$ preserves filtered colimits.
\end{proof}

Let $\Fun_{\mc{V}}(\mc{C},\mc{D})$ denote the category of \mc{V}\=/enriched functors between \mc{V}\=/closed modules and morphisms given by \mc{V}\=/enriched natural transformations.
\begin{proposition}\label{prop:filt-colim-good} 
Let \mc{V} be a confined symmetric monoidal category and \mc{M} a confined \mc{V}\=/module.
Then the full subcategory of docile functors in $\Fun_{\mc{V}}(\mc{M},\mc{M})$ is closed under finite limits, filtered colimits and composition.
\end{proposition}

\begin{proof}
We need to show that a filtered colimit of docile functors is left exact and preserves compact cotensors. We start by showing left exactness.
Let \mc{F} be a filtered category and \mc{D} be a finite category and $G$ a $\mc{F}\times\mc{D}$\=/diagram of docile functors. We need to show that the canonical map
\[
\colim_{\mc{F}}\lim_{\mc{D}} G(f,d) \xto{\sigma} \lim_{\mc{D}}\colim_{\mc{F}} G(f,d)
\]
is an isomorphism in $\Fun_{\mc{V}}(\mc{M},\mc{M})$. By Proposition~\ref{prop:Heine-rules} the forgetful functor to $\Fun(\mc{M},\mc{M})$ preserves limits and colimits and sends $\sigma$ to the corresponding canonical map $\sigma'$ in $\Fun(\mc{M},\mc{M})$. Since colimits are computed objectwise in $\Fun(\mc{M},\mc{M})$, and filtered colimits are left exact by Lemma~\ref{lem:filt-colim-lex}, the map $\sigma'$ is invertible. But the functor forgetting the enrichment is conservative by Proposition~\ref{prop:Heine-rules}, it follows that $\sigma$ is invertible.

We proceed to show that filtered colimits commute with compact cotensors. The comparison map $\colim_{\mc{F}}[C,G(f)]\to[C,\colim_{\mc{F}}G(f)]$ for a compact $C$ in \mc{V} is an isomorphism if and only if for every compact object $K$ of \mc{V} the induced map
\[
[K,\colim_{\mc{F}}[C,G(f)]]\to[K,[C, \colim_{\mc{F}}G(f)]]
\]
is an isomorphism. The cotensor $[K,-]$ commutes with filtered colimits by Lemma~\ref{lem:compact-cotensors-good}.
Now the isomorphism can be checked directly.
We have shown that the subcategory of docile functors is closed under filtered colimits. 

Closure under finite limits is similar, but easier. It is left to the reader.
It is clear that the composition of docile functors is again docile. 
\end{proof}

\begin{construction}\label{def:PZ} 
We fix a map $z:Z\to \one$ in \mc{V}, where $\one$ is the monoidal unit and $Z$ is an object of \mc{V}. We obtain a functor 
\[
T:=[Z,-]:\mc{V}\to \mc{V}
\]
and a natural transformation $t:=[z,-]:\Id\to T$.
Now let $P:\mc{V} \to \mc{V}$ be the colimit of the sequence 
\[
P:=\colim \bigl( \Id\xto{t} T\xto{tT} T^2\xto{tT^2} T^3\xto{tT^3} T^4\to \hdots\ \bigr)
\]
in the category $\Fun_{\mc{V}}(\mc{V},\mc{V})$ and let $p:\Id \to P$ be the canonical natural transformation.
\end{construction}

\begin{lemma}\label{lem:Qisgood} 
Let \mc{V} be a confined symmetric monoidal category.
In the situation of Construction~\ref{def:PZ} the following holds:
\begin{enumerate}[label={\rm (\roman*)}]
   \item\label{lem:Qisgood:1} 
The functor $T$ is a \mc{V}\=/functor and $t$ is a \mc{V}\=/enriched natural transformation. 
   \item\label{lem:Qisgood:2}
The (underlying) functor $P$ is isomorphic to the objectwise colimit of the defining sequence above. It is a \mc{V}\=/functor and the natural transformation $p$ is \mc{V}\=/enriched. 
\end{enumerate}
If the object $Z$ is compact, then the endofunctors $T$ and $P$ are docile. 
\end{lemma}

\begin{proof}
Lemma~\ref{fact:1} gives~\ref{lem:Qisgood:1}. 

As a composition of \mc{V}\=/functors, for all $n$ the functor $T^n$ is a \mc{V}\=/functor and the natural transformations in the defining diagram of $P$ are \mc{V}\=/enriched. Hence, the defining colimit of $F$ is indeed a diagram in the category $\Fun_{\mc{V}}(\mc{V},\mc{V})$ of \mc{V}\=/enriched functors and \mc{V}\=/enriched natural transformations. 
Note that \mc{V} is cocomplete since it is a presentable category.
Thus the claim~\ref{lem:Qisgood:2} follows entirely from Proposition~\ref{prop:Heine-rules}. The colimit $P$ exists as a \mc{V}\=/functor, but can also be computed objectwise. The map $p$ is a colimit of \mc{V}\=/enriched natural transformations and is therefore itself \mc{V}\=/enriched.

Lemma~\ref{lem:compact-cotensors-good} says that $T=[Z,-]$ is a docile functor if $Z$ is compact. The remaining claims follows from Proposition~\ref{prop:filt-colim-good} and Example~\ref{exam:module-over-itself}. For all $n$ the functor $T^n$ is docile as a composition of docile functors.
The functor $P$ is docile since it is a sequential (filtered) colimit of docile functors in the category $\Fun_{\mc{V}}(\mc{V},\mc{V})$.
\end{proof}

\section{Tidy maps and their localizations}
\label{sec:Quahog}

\subsection{Tidy maps and the Weiss trick}

The following idea is due to Weiss~\cite[e.12]{Weiss98}. It leads us to the notion of a tidy map in Definition~\ref{BWlemmadef} and will be used in Lemma~\ref{rem-nat-trans25}.

\begin{proposition}[Weiss trick]
\label{prop:Weiss-trick} 
Let \mc{V} be a symmetric monoidal closed category, $R:\mc{V}\to \mc{V}$ a \mc{V}\=/functor and $r:\Id\to R$ a \mc{V}\=/enriched natural transformation. For a map $f:A\to B$ in \mc{V} the square 
\begin{equation} \label{Weisslemma}
\begin{tikzcd}
{[B,X]} \ar[rr,"{[B,rX]}"] \ar[d,"{[f,X]}"'] && {[B,RX]} \ar[d,"{[f, RX]}"]
 \\
{[A,X]} \ar[rr,"{[A,rX]}"']  && {[A,RX]}
\end{tikzcd}
\end{equation}
commutes for every object $X$ in \mc{V}. If the map $Rf:RA\to RB$ is invertible, then the square has a diagonal filler $\delta(X):[A,X]\to [B,RX]$ which is \mc{V}\=/enriched natural in $X$. 
\end{proposition} 

\begin{proof} 
Both $R$ and $r$ are \mc{V}\=/enriched. Therefore the square~\eqref{Weisslemma} commutes. Let $\theta$ denote the \mc{V}\=/enrichment of $R$. Then the square 
\begin{equation*} 
\begin{tikzcd}
{[B,X]} \ar[rr,"{\theta(B,X)}"] \ar[d,"{=}"'] && {[RB,RX]} \ar[d,"{[rB, RX]}"]
 \\
{[B,X]} \ar[rr,"{[B,rX]}"']  && {[B,RX]}
\end{tikzcd}
\end{equation*}
commutes as it expresses the fact that $r:\Id\to R$ is a \mc{V}\=/enriched natural transformation.
Thus, $[B,rX]=[rB,RX]\circ\theta (B,X)$, and similarly $[A,rX]=[rA,RX]\circ\theta (A,X)$.
Hence the square~\eqref{Weisslemma} can be factored in the following way:
\begin{equation*} 
\begin{tikzcd}
{[B,X]} \ar[r,"{\theta}"] \ar[d,"{[f,X]}"']  &{[RB,RX]}  \ar[d,"{[Rf,RX]}"']   \ar[rr,"{[rB,RX]}"]                  && {[B,RX]}\ar[d,"{[f,RX]}"] 
 \\
{[A,X]} \ar[r,"{\theta}"']  & {[RA,RX]}  \ar[rr,"{[rA,RX]}"'] \ar[u, "{[g,RX]}"', bend right] &&{[A,RX]}\,.
\end{tikzcd}
\end{equation*}
But the middle map $[Rf,RX]$ is invertible, since $Rf$ is invertible by hypothesis. It follows that the composed square has a diagonal filler.
More precisely, if $g:=(Rf)^{-1}$, then $[g,RX]\cong [Rf,RX]^{-1}$ and the map 
\[
\delta(X)=[rB,RX]\circ [g,RX]\circ \theta(A,X):[A,X]\to [RA,RX]\to [RB,RX] \to [B,RX]
\]
is a diagonal filler of the square~\eqref{Weisslemma}. Moreover, note that $g$, as an inverse of an enriched natural transformation, is itself \mc{V}\=/enriched. Hence the map $\delta(X)$ is a \mc{V}\=/natural transformation in $X$, since it is a composition of \mc{V}\=/natural transformations.
\end{proof}

\begin{definition}[Tidyness]
\label{BWlemmadef} 
Suppose that the symmetric monoidal category \mc{V} is confined.
We will say that a map $z:Z\to \one$ in \mc{V} is {\it tidy} if the object $Z$ is compact and the map $Pz:PZ\to P\one$ is invertible, where $P$ is the endofunctor of \mc{V} from Construction~\ref{def:PZ}.
\end{definition}

The invertibility condition is of course inspired by the Weiss trick in the previous lemma and it is going to be essential in the proofs of Lemma~\ref{rem-nat-trans25}, Lemma~\ref{BWlemma}, Lemma~\ref{lem:tensorlocalnew} and consequently in the proof of the main Theorem~\ref{thm:monoidal-localization}. The compactness of $Z$ is necessary to ensure that the functors $T$ and $P$ commute with filtered colimits. This will be used in the constructions of Diagrams~\eqref{thetPdiagram} and~\eqref{basic-conePTn}, in the proofs of Lemma~\ref{BWlemma}, Proposition~\ref{prop:pPandoO} and everything onwards.

\begin{example}
Not every map of the form $Z\to \one$, $Z$ compact, is tidy. Take the category of spaces as a cartesian closed category. It is confined. The map $S^0\to 1$ is not tidy: the associated cotensor is $T(X)=[S^0,X]=X\times X$, and for the resulting $P=\colim_n T^n$ we have $P(S^0)\neq 1=P(1)$. This $P$ is not idempotent and therefore not a localization. In fact, in $(\mc{S},\times,1)$ the only tidy maps are $0\to 1$ and $1=1$. This follows since \mc{S} admits only two left exact localizations: the one that inverts every map and the identity functor of \mc{S}.
\end{example}

We expect tidy maps to be rare. Because of the self-referential nature of their definition, it is usually difficult to decide whether a given map is tidy or not. However, see Propositions~\ref{r-star-is-On-hence-Pn-equivalence137} and~\ref{prop:Pn-zetan-WeissA}.

\begin{lemma}  \label{rem-nat-trans25} 
If \mc{V} is confined and $z:Z\to \one$ is tidy, the square 
\begin{equation*} 
  \begin{tikzcd}[bo column sep=large]
   \Id \ar[r, "p"] \ar[d, "t"']   & P\ar[d, "tP"]  \\
    T \ar[r, "Tp" ]   & TP 
   \end{tikzcd}
   \end{equation*}
of \mc{V}\=/natural transformations has a \mc{V}\=/enriched diagonal filler $T\to P$. 
\end{lemma}

\begin{proof} 
Apply Proposition~\ref{prop:Weiss-trick} to $(z:Z\to \one)=(A\to B)$ and $R=P$.
\end{proof}

\subsection{Constructing the localization}
\label{subsec:constructing-loc}

Let \mc{V} be a confined symmetric monoidal category. We are heading towards Theorem~\ref{thm:monoidal-localization} stating that $P$ from Construction~\ref{def:PZ} is a symmetric monoidal left exact localization of \mc{V} under the assumption that the map $Z\to \one$ is tidy. First, it is necessary to examine carefully the map $tP$ with the goal of showing that it is invertible in Lemma~\ref{BWlemma}.

By construction, we have a colimit cone
\begin{equation} \label{basic-coneP}
\begin{tikzcd}[bo column sep=large]
{\Id} \ar[r,"{t}"] \ar[dddrrr,"{p:=p_0}"'] & 
{T} \ar[r,"{tT}"] \ar[dddrr,"{p_{1}}"] & 
{T^{2}} \ar[r,"{tT^2}"] \ar[dddr,"{p_{2}}"] &
{T^{3}} \ar[r,"{tT^3}"] \ar[ddd,"{p_{3}}"] & 
{T^{4}}  \ar[dddl,"{p_{4}}"] \ar[r,"{tT^4}"] & \cdots   \\\
&&&& & \\
&&&& & \\
&&& P \,.
\end{tikzcd}
\end{equation}
with conical maps $p_n:T^n\to P$ that are $ \mc{V}$\=/natural transformations for every $n\geq 0$. 
In preparation for the proof of Theorem~\ref{thm:monoidal-localization} we postcompose Diagram~\eqref{basic-coneP} by $T$:
\begin{equation*}
\begin{tikzcd}[bo column sep=large]
{T} \ar[r,"{Tt}"] \ar[dddrrr,"{Tp=Tp_0}"'] & 
{T^2} \ar[r,"{TtT}"] \ar[dddrr,"{Tp_{1}}"] & 
{T^{3}} \ar[r,"{TtT^2}"]  \ar[dddr,"{Tp_{2}}"] &
{T^{4}} \ar[r,"{TtT^3}"] \ar[ddd,"{Tp_{3}}"] & 
{T^{5}}  \ar[dddl,"{Tp_{4}}"] \ar[r,"{TtT^4}"] & \cdots   \\\
&&&& & \\
&&&& & \\
&&& TP \,.
\end{tikzcd}
\end{equation*}
Putting the previous diagram back to back with Diagram~\eqref{basic-coneP} we obtain the following diagram that commutes by naturality of the map $t:\Id \to T$:
\begin{equation}
\begin{tikzcd} \label{thetPdiagram}
&&&&&  P \ar[ddddd,"{tP}"]  \\
&&&&&  \\
\Id \ar[d,"{t}"'] \ar[r,"{t}"] \ar[uurrrrr, "p_0", bend left] 
& T\ar[r,"{tT}"] \ar[d,"{tT}"'] \ar[uurrrr, "p_1", bend left] 
& T^2 \ar[d,"{tT^2}"'] \ar[r,"{tT^2}"]  \ar[uurrr, "p_2", bend left]  
& T^3 \ar[d,"{tT^3}"'] \ar[uurr, "p_3", bend left] \ar[r] 
& \cdots &
 \\
T \ar[ddrrrrr, "Tp=Tp_0"', bend right]  \ar[r,"{Tt}"'] 
& T^2 \ar[r,"{TtT}"'] \ar[ddrrrr, "Tp_1", bend right]  
& T^3 \ar[r,"{TtT^2}"']  \ar[ddrrr, "Tp_2", bend right]
& T^4 \ar[ddrr, "Tp_3", bend right] \ar[r]
& \cdots &   \\
&&&&&  \\
&&&&&  TP\,.
\end{tikzcd}
\end{equation}

\begin{lemma}  \label{lem:tP-as-colim-of-tTk} 
If $Z$ is compact in \mc{V}, then the natural transformation $tP:P\to TP$ is the filtered  colimit of the natural transformations $tT^n: T^n\to T^{n+1}$ of Diagram~\eqref{thetPdiagram}.
\end{lemma} 

\begin{proof}
The top cone of the diagram is a colimit cone by definition of $P$.
The bottom cone is obtained by applying the functor $T$ to the top cone.
But the functor $T$ preserves filtered  colimits by Lemma~\ref{lem:Qisgood}, since $Z$ is compact. It follows that the bottom cone is also a colimit cone. Hence the map $tP$ is the colimit
of the sequence of maps $tT^n: T^n\to T^{n+1}$.
\end{proof}

In preparation for the proof of Lemma~\ref{BWlemma} we precompose Diagram~\eqref{basic-coneP} with $T^n$ to obtain a new colimit diagram:
\begin{equation} 
\label{basic-conePTn}
\begin{tikzcd}[bo column sep=large]
&&& PT^n & & \\
&&&& & \\
&&&& & \\
&&&& & \\
{T^n} \ar[r,"{tT^n}"'] \ar[uuuurrr,"{pT^n=p_0T^n}"] & 
{T^{n+1}} \ar[r,"{tT^{n+1}}"'] \ar[uuuurr,"{p_{1}T^n}"'] & 
{T^{n+2}} \ar[r,"{tT^{n+2}}"']  \ar[uuuur,"{p_{2}T^n}"'] &
{T^{n+3}} \ar[r,"{tT^{n+3}}"'] \ar[uuuu,"{p_{3}T^n}"'] & 
{T^{n+4}}  \ar[uuuul,"{p_{4}T^{n+4}}"'] \ar[r,"{tT^{n+4}}"'] & \cdots   
\end{tikzcd}
\end{equation}
Observe that the bottom line of Diagram~\eqref{basic-conePTn} is a cofinal sequence of the top line of Diagram~\eqref{basic-coneP}. It follows that there is a unique isomorphism $\sigma_n: PT^n\to P$ such that 
\begin{equation} \label{colimshift}
\sigma_n p_kT^{n}=p_{n+k}\,,
\end{equation}
for every $k\ge 0$. This is depicted in the diagram:
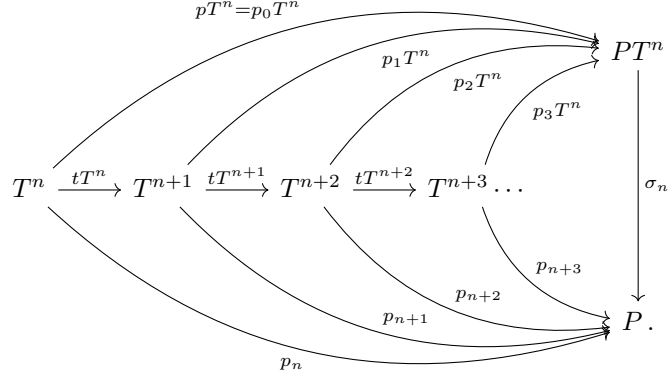
\[
\begin{tikzcd} 
&&&&  PT^n \ar[dddd,"{\sigma_n}"]  \\
&&&&  \\
T^n \ar[ddrrrr, "  p_n "', bend right]  \ar[r,"{tT^n}"] \ar[uurrrr, "pT^n=p_0T^n", bend left] 
& T^{n+1}  \ar[r,"{tT^{n+1}}"] \ar[ddrrr, "p_{n+1} ", bend right] \ar[uurrr, "p_1T^n"', bend left]  
& T^{n+2}\ar[r,"{tT^{n+2}}"]  \ar[ddrr, "p_{n+2}", bend right] \ar[uurr, "  p_2T^n "', bend left]  & T^{n+3} \ar[ddr, "p_{n+3}", bend right]\ar[uur, " p_3T^n "', bend left] \cdots &  \\
&&&&  \\
&&&&  P\,.
\end{tikzcd}
\]
Note that the map $\sigma_n$ goes "backwards" and needed to be constructed carefully. It is a crucial ingredient in the next proof. 

\begin{lemma} \label{BWlemma} 
Let \mc{V} be confined and $z:Z\to \one$ tidy. Then the \mc{V}\=/natural transformation $tP: P\to TP$ is invertible.
\end{lemma} 

\begin{proof}
By Lemma~\ref{lem:tP-as-colim-of-tTk}, this map is the filtered colimit of the Diagram~\eqref{thetPdiagram} of maps $tT^n: T^n\to T^{n+1}$. 
By Lemma~\ref{Whiteheadcolimit17}, we can prove that the map $tP:P\to TP$ is invertible by showing that the square
\begin{equation*} 
  \begin{tikzcd}[bo column sep=large]
   T^n \ar[r, "p_n"] \ar[d, "tT^n"']   & P\ar[d, "tP"] \\
    T^{n+1} \ar[r, "Tp_n" ]   & TP
  \end{tikzcd}
\end{equation*}
has a diagonal filler for every $n\geq 0$.
If $\sigma_n:PT^n\to P$ is the isomorphism defined in Equation~\eqref{colimshift}, then we have $p_n=\sigma_n (pT^n)$. Hence the square above is the composite of the two commutative squares:
\begin{equation*}
  \begin{tikzcd}[bo column sep=large]
   T^n \ar[r, "{pT^n}"] \ar[d, "{tT^n}"']   & 
   PT^n\ar[d, "{tPT^n}"] \ar[r,"{\sigma_n}"] & P \ar[d,"{tP}"] \\
       T^{n+1} \ar[r, "{TpT^n}" ]   & TPT^n \ar[r,"{T\sigma_n}"] & TP\,.
  \end{tikzcd}
\end{equation*}
But the left hand square of this diagram is obtained by precomposing the square in Lemma~\ref{rem-nat-trans25} with $T^n$. Hence the left hand square has a diagonal filler, since the square in Lemma~\ref{rem-nat-trans25} has a diagonal filler.
It follows that the composite square has a diagonal filler for every $n\geq 0$ proving that $tP:P\to TP$ is invertible. 
\end{proof}

\begin{lemma}\label{lem:com570}
Suppose the \mc{V}\=/functor $F:\mc{V}\to \mc{V}$ preserves compact cotensors. 
Then we have a commutative diagram 
\begin{equation*} 
\begin{tikzcd}
& F \ar[dl,"Ft"']\ar[dr,"tF"] \\
FT  \ar[rr,"\cong"',"\gamma"] && TF  
\end{tikzcd}
\end{equation*}
of \mc{V}\=/natural transformations, where $\gamma=\gamma(Z,X):F[Z,X] \to  [Z,FX]$ is the coassembly map of the functor $F$. 
\end{lemma}  

\begin{proof} 
From the map $z:Z\to \one$ we obtain a commutative square of \mc{V}\=/natural transformations in the variable $X$ in \mc{V}:
\begin{equation*} 
\begin{tikzcd}
{F[\one,X]}  \ar[d,"{F[z,X]}"']  \ar[rrr,"{\gamma(\one,X)=\id_{FX}}"]
&&&  {[\one,FX]}  \ar[d,"{[z,FX]}"]  \\
{F[Z,X]}  \ar[rrr,"{\gamma(Z,X)}"] &&& {[Z,FX]}\,.
\end{tikzcd}
\end{equation*}    
The coassembly map $\gamma(Z,X):F[Z,X] \to  [Z,FX]$ is invertible, since the functor $F$ preserves compact cotensors and $Z$ is compact.
\end{proof}

\begin{proposition}\label{prop:pPandoO}
If \mc{V} is confined and $z:Z\to \one$ is tidy, 
then the natural transformations 
\[
p P:P\to P^2\ \qquad{\rm and}\qquad\ Pp:P\to P^2
\]
invertible.
\end{proposition}

\begin{proof} 
Let us show that the natural transformation $p P:P\to P^2$ is invertible.
If we precompose the defning colimit of $P$ with $P$, we obtain a colimit
\begin{equation*} 
\colim\left(
\begin{tikzcd}[bo column sep=large]
{P} \ar[r,"{tP}"] & 
{TP} \ar[r,"{tTP}"] & 
{T^{2}P} \ar[r,"{tT^2P}"]  &
{T^{3}P} \ar[r,"{tT^3P}"] & \cdots  
\end{tikzcd}
\right) = P^2\,
\end{equation*}
and the transfinite composition of the maps in the colimit is $pP:P\to P^2$.
The map $tT^n$ is isomorphic to the map $T^nt$, since the symmetric monoidal structure on \mc{V} yields (canonical) isomorphisms between the maps $z\otimes Z^{\otimes n}$ and $Z^{\otimes n}\otimes z$.
Hence the map $tT^nP$ is isomorphic to $T^ntP$, which is invertible, since the map $tP$ is invertible by Lemma~\ref{BWlemma}.
So $pP$ is invertible, since the transfinite composition of isomorphisms is an isomorphism. 

Let us show that the map $Pp:P\to P^2$ is invertible.
The functor $P$ preserves filtered colimits by Lemma~\ref{lem:Qisgood}.
If we postcompose the defining colimit of $P$ with $P$, we therefore obtain a colimit
\begin{equation*} 
\colim\left(
\begin{tikzcd}[bo column sep=large]
{P} \ar[r,"{Pt}"] & 
{PT} \ar[r,"{PtT}"] & 
{PT^{2}} \ar[r,"{PtT^2}"]  &
{PT^{3}} \ar[r,"{PtT^3}"] & \cdots
\end{tikzcd}
\right) = P^2\,.
\end{equation*}

The map $tP:P\to TP$ is invertible by Lemma~\ref{BWlemma}, since the map $z:Z\to \one$ is tidy. 
It follows that the map $Pt:P\to PT$ is invertible by Lemma~\ref{lem:com570}, since $P$ is docile.
Hence the map $PtT^{n}:PT^n\to PT^{n+1}$  is invertible for every $n\geq 0$.
So $P^2$ is the colimit of an increasing sequence of isomorphisms and the conical map $Pp$ is invertible.
\end{proof}

\begin{definition}\label{def:STPQ-closed}
Let $L:\mc{C}\to\mc{C}$ be an endofunctor of a category \mc{C} with a coaugmentation $\ell:\Id\to L$. 
We say that an object $X$ in $\mc{C}$ is $L$\=/{\it closed} if the map $\ell X:X\to LX$ is invertible.
We say that a map $f:X\to Y$ in \mc{C} is {\it $L$\=/\it closed} if the following naturality square is cartesian:
\[
\begin{tikzcd}
    X \ar[r, "\ell X"]\ar[d, "f"'] & LX \ar[d, "Lf"] \\
    Y \ar[r, "\ell Y"] & LY\,.
\end{tikzcd}
\]
Let $\mc{C}^L$ denote the full subcategory of \mc{C} formed by $L$\=/closed objects.
\end{definition}

\begin{lemma}\label{lem:Tn-local-Pn-local-mapsA}
Let \mc{V} be confined and $z:Z\to \one$ tidy. Then we have:
\begin{enumerate}[label={\rm (\roman*)}]
  \item\label{lem:Tn-local-Pn-local-mapsA:1}
A map $f:X\to Y$ in $ \mc{V}$  is $T$\=/closed if and only if it is $P$\=/closed.
  \item\label{lem:Tn-local-Pn-local-mapsA:2}
An object in $ \mc{V}$ is $T$\=/closed if and only if it is $P$\=/closed.
  \item \label{lem:Tn-local-Pn-local-mapsA:3}
The object $PX$ is $P$\=/closed for every $X$ in \mc{V}. 
\end{enumerate}
\end{lemma}

\begin{proof} 
To prove~\ref{lem:Tn-local-Pn-local-mapsA:1}, let $f:X\to Y$ be $P$\=/closed. Consider the commutative cube:
  \[
  \begin{tikzcd}[bo column sep=large]
    & PX \ar[rr, "tPX"] \ar[dd, "Pf"', near end] & & TPX \ar[dd, "TPf"] \\
    X \ar[rr, crossing over,"tX", near end] \ar[dd,"f"'] \ar[ur, "pX", near end] & & TX \ar[ur, "TpX"', near start] & \\
    & PY \ar[rr, "tPY", near start] & & TPY \\
    Y \ar[rr, "tY"] \ar[ur, "pY", near end] & & T
    Y \ar[ur, "TpY"', near start]\ar[from=uu, crossing over] & 
  \end{tikzcd}
  \]
The left hand face of the cube is cartesian, since $f$ is $P$\=/closed. The right hand face is also cartesian, since the functor $T$ preserves limits. But the horizontal maps of the back face, $tPX$ and $tPY$, are isomorphism by Lemma~\ref{BWlemma}. It follows that the front face is cartesian. Thus, $f$ is $T$\=/closed. 

Conversely, let $f:X\to Y$ be $T$\=/closed. We need to show that the following square is cartesian:
\begin{equation*} 
\begin{tikzcd} 
   X \ar[r, "pX"] \ar[d, "f"']   & PX \ar[d, "Pf"] \\
   Y \ar[r, "pY" ]   & PY\,.
   \end{tikzcd}
   \end{equation*}
But this square is the infinite composition of the sequence:
\begin{equation*}
\begin{tikzcd}
X \ar[d,"{f}"'] \ar[r,"{tX}"] &TX \ar[r,"{tTX}"] \ar[d,"{Tf}"'] &T^2X \ar[d,"{T^2f}"'] \ar[r,"{tT^2X}"]  & T^3X
\ar[d,"{T^3f}"'] \ar[r]  & \cdots
 \\
Y  \ar[r,"{tY}"'] & TY  \ar[r,"{tTY}"'] & T^2Y \ar[r,"{tT^2Y}"'] & T^3Y \ar[r]  & \cdots
\end{tikzcd}
\end{equation*}
It suffices to show that each square in the sequence is cartesian, since filtered colimits preserves finite limits in \mc{V}.
An individual square looks like this:
\begin{equation}\label{eq:square-n} 
  \begin{tikzcd} 
     T^{n} X \ar[rr, "{tT^{n} X}"] \ar[d, "{T^{n}f}"] && T(T^{n}X) \ar[d, "{T(T^{n}f)}"] \\
     T^{n}T \ar[rr, "{tT^{n}Y}"] && T(T^{n}Y) 
  \end{tikzcd} 
\end{equation} 
The case $n=0$ is clear, since the map $f$ is $T$\=/closed by assumption.
In general the square above is the composition: 
\begin{equation*}
  \begin{tikzcd} 
    T^{n}X \ar[rr, "{T^{n}tX}"] \ar[d, "{T^{n}f}"] && T^{n}(TX) \ar[d, "{T^{n}(Tf)}"]  \ar[rr, "\gamma "] &&  T(T^{n}X)\ar[d, "{T(T^{n}f)}"] \\   
    T^{n}Y \ar[rr, "{T^{n}tY}"] && T^{n}(TY)  \ar[rr, "\gamma "] &&  T(T^{n}Y)\,,
  \end{tikzcd}
\end{equation*}
where $\gamma$ is the coassembly isomorphism from Lemma~\ref{lem:com570} with $F:=T^n$. The left hand square is the image by $T^n$ of the case $n=0$ considered before. 
So the left hand square is cartesian, since the functor $T^n$ preserves limits.  
The right hand square is also cartesian, since its horizontal maps are invertible.
It follows that for all $n\ge 0$ the square~\eqref{eq:square-n} is cartesian. 
Hence the map $f$ is $P$\=/closed.

Statement~\ref{lem:Tn-local-Pn-local-mapsA:2} is a special case of~\ref{lem:Tn-local-Pn-local-mapsA:1}.

For~\ref{lem:Tn-local-Pn-local-mapsA:3} observe that $PX$ is $T$\=/closed by Lemma~\ref{BWlemma}. Thus, $PX$ is $P$\=/closed by~\ref{lem:Tn-local-Pn-local-mapsA:2}. This can also be deduced from Proposition~\ref{prop:pPandoO} as the condition is equivalent to the fact that $pP$ is an isomorphism.
\end{proof}

\begin{definition}
We say that a map $\rho:X\to Y$ is a {\it \mc{V}\=/reflection into the full subcategory of $P$\=/closed objects $\mc{V}^P$} if $Y$ belongs to $\mc{V}^P$ and the map $[\rho,W]:[Y,W]\to [X,W]$ is invertible for every $W$ in $ \mc{V}^P$.
\end{definition}

\begin{lemma}\label{lem:tensorlocalnew} 
Suppose that \mc{V} is confined and that $z:Z\to \one$ is tidy. Then the map $pX:X \to PX$ is a \mc{V}\=/reflection into $\mc{V}^P$ for every object $X$ in \mc{V}.
\end{lemma}

\begin{proof} By Lemma~\ref{lem:Tn-local-Pn-local-mapsA} $PX$ is in $\mc{V}^P$. It remains to show that the map $[pX,W]:[PX,W]\to [X,W]$ is invertible for every $P$\=/closed $W$. The square 
\begin{equation*}  
  \begin{tikzcd} 
   {[PX,W]} \ar[rr, "{[PX, pW]}"] \ar[d, "{[pX,W]} "']   && {[PX,PW]} \ar[d, "{[pX, PW]} "] \\
  {[X,W]}\ar[rr, "{[X,pW]}" ]  & & {[X,PW]}
   \end{tikzcd}
\end{equation*}  
commutes by applying the natural transformation $[pX,-]$ to the map $pW$. The Weiss trick (Proposition~\ref{prop:Weiss-trick}) can be applied to this square.
For the benefit of the reader we supply the factorization of the square:
\begin{equation*}  
  \begin{tikzcd} 
   {[PX,W]} \ar[rr, "{\theta}"] \ar[d, "{[pX,W]} "'] && {[P^2X,PW]} \ar[rr, "{[pX,PW]}"]\ar[d, "{[P(pX),PW]}"', "\cong"]  && {[PX,PW]} \ar[d, "{[pX, PW]}"] \\
  {[X,W]}\ar[rr, "{\theta}" ]  && {[PX,PW]} \ar[rr, "{[pX,PW]}"] && {[X,PW]}\,.
   \end{tikzcd}
\end{equation*}
Here $\theta$ is (part of) the \mc{V}\=/enrichment of $P$.  
The map $P(pX)$ is invertible by Proposition~\ref{prop:pPandoO}. Using the inverse of the middle vertical map, the square above has a diagonal filler $d$ 
\begin{equation*}  
  \begin{tikzcd}
   {[PX,W]} \ar[rr, "{[PX, pW]}"] \ar[d, "{[pX,W]} "']   && {[PX,PW]} \ar[d, "{[pX, PW]} "] \\
  {[X,W]}\ar[rr, "{[X,pW]}"' ] \ar[rru, "d" ]  & & {[X,PW]}
   \end{tikzcd}
\end{equation*}
by Proposition~\ref{prop:Weiss-trick}. But $pW$ is invertible, since $W$ in $\mc{V}^P$. It follows that all maps in this diagram are invertible proving that the map $[pX,W]$ is invertible. 
\end{proof}

\begin{definition} 
We will say that a map $f:X\to Y$ in \mc{V} is a $P$\=/{\it equivalence}
if the map $P(f): PX\to PY$ is invertible.
\end{definition}

\begin{lemma} \label{P-equiv}
A map $f:X\to Y$ is a $P$\=/equivalence if and only if the map $[f,W]:[Y,W]\to [X,W]$ is invertible for every $W$ in $\mc{V}^P$.
\end{lemma}

\begin{proof}
By definition $f$ is a $P$\=/equivalence if and only if the map $P(f)$ is invertible. The image of the commutative square
\begin{equation*}  
   \begin{tikzcd}[bo column sep=large]
   {X} \ar[r, "{pX}"] \ar[d, "{f}"']  & PX\ar[d, "{P(f)}"] \\
   {Y} \ar[r, "{pY}" ]  & {PY}
    \end{tikzcd}
   \end{equation*}
by the contravariant functor $[-,W]$ is the commutative square:
\begin{equation*} 
   \begin{tikzcd}[bo column sep=large]
    {[X,W]}                     && \ar[ll, "{[pX,W]}"']  [PX,W] \\
    {[Y,W]} \ar[u, "{[f,W]}"]   && \ar[ll, "{[pY,W]}"] \ar[u, "{[P(f),W]}"'] [PY,W]  \,.
   \end{tikzcd}
\end{equation*} 
By Yoneda, the map $P(f)$ in $\mc{V}^P $ is invertible if and only if the map $[P(f),W]$ in \mc{V} is invertible for every object $W$ in $\mc{V}^P$.
But the horizontal maps of the lower square are invertible, since the maps $pX:X\to PX$ and $pY:Y\to PY$ are \mc{V}\=/reflecting into $\mc{V}^P$ by Lemma~\ref{lem:tensorlocalnew}. Hence the map $P(f)$ is invertible if and only if the map $[f,W]$ is invertible  for every object $W$ in $\mc{V}^P$.
\end{proof}

\begin{lemma}\label{lem:P-equivalences}
For any $P$\=/equivalence $f$ and any object $V$ the map $V\otimes f$ is a $P$\=/equivalence. 
\end{lemma}

\begin{proof}
Let $V$ be an object in \mc{V} and $W$ be in $\mc{V}^P$. Then the map $[f,W]$ is invertible by Lemma~\ref{P-equiv}. Hence the map $[V\otimes f,W]=[V,[f,W]]$ is invertible. So, again by Lemma~\ref{P-equiv}, the map $V\otimes f$ is a $P$\=/equivalence. 
\end{proof}

From this lemma it follows that for any two $P$\=/equivalences $f,f'$ in \mc{V} their tensor product $f\otimes f'$ is a $P$\=/equivalence since $f\otimes f'=(f\otimes Y')(X\otimes f')$.
  
\begin{theorem}\label{thm:monoidal-localization}
Let \mc{V} be a confined symmetric monoidal category and let $z:Z\to \one$ be a tidy map. Define $T:=[Z,-]$, $t:=[z,-]:\Id \to T$, $P=\colim_n T^n$ and $p:\Id \to P$ as in \emph{Construction~\ref{def:PZ}}. Let $\mc{V}^P$ be the subcategory of $P$\=/closed objects of \mc{V} and let $\comp(\mc{V})$ be the subcategory of compact objects of \mc{V}. Then:
\begin{enumerate}[label={\rm (\roman*)}]
  \item\label{thm:monoidal-localization:1}
An object of \mc{V} is $T$\=/closed if and only if it is $P$\=/closed.
  \item\label{thm:monoidal-localization:2}
For every $A$ in \mc{V} and $X$ in $\mc{V}^P$ the cotensor $[A,X]$ is in $\mc{V}^P$.
  \item\label{thm:monoidal-localization:3}
The subcategory $\mc{V}^P$ is \mc{V}\=/reflective, and the natural transformation $p:\Id\to P$ is \mc{V}\=/reflecting into $\mc{V}^P$. The reflector $P:\mc{V}\to \mc{V}^P$ is \mc{V}\=/left exact.
  \item\label{thm:monoidal-localization:4}
The category $\mc{V}^P$ is symmetric monoidal closed with tensor $X\otimes_P Y:=P(X\otimes Y)$ for every $X,Y$ in $ \mc{V}^P$. Its unit object is $\one_P=P(\one)$. The localization functor $P:\mc{V} \to \mc{V}^P$ is symmetric monoidal.
  \item\label{thm:monoidal-localization:5}
The symmetric monoidal category $(\mc{V}^P,\otimes_P, P(\one))$ is confined and the reflector $P:\mc{V} \to \mc{V}^P$ is confined. 
Every compact object of $\mc{V}^P$ is a retract of an object in $P(\comp(\mc{V}))$. The subcategory $P(\comp(\mc{V}))$ is dense in $\mc{V}^P$.
\end{enumerate}
\end{theorem}
 
\begin{proof}
\ref{thm:monoidal-localization:1} This was proved in Lemma~\ref{lem:Tn-local-Pn-local-mapsA}.                     

\smallskip
\noindent\ref{thm:monoidal-localization:2} Let $X$ be $T$\=/closed. We will show that the object $[A,X]$ is $T$\=/closed for every $A$ in \mc{V}. The claim then follows from~\ref{thm:monoidal-localization:1}. Consider the diagram:
\begin{equation*} 
\begin{tikzcd}
{[\one,[A,X]]} \ar[d,"{t[A,X]=[z,[A,X]]}"']  \ar[r,"\cong"]
& {[\one \otimes A,X]} \ar[d,"{[z\otimes A,X]}"']  \ar[r,"\cong"]
& {[A\otimes \one ,X]} \ar[r,"\cong"]   \ar[d,"{[A\otimes z,X]}"']
& {[A,[\one,X]]}    \ar[d,"{[A, [z,X]]=[A, tX]}"]   \\   
 {[Z,[A,X]]}  \ar[r,"\cong"]
& {[Z \otimes A,X]}   \ar[r,"\cong"]
& {[A\otimes Z ,X]} \ar[r,"\cong"]
& {[A,[Z,X]]}\,.
\end{tikzcd}
\end{equation*} 
The horizontal maps are coassembly maps $\gamma(\one,X)$ and $\gamma(Z,X)$. They are invertible, since the functor $[A,-]$ preserves all cotensors as explained in the proof of Lemma~\ref{lem:compact-cotensors-good}. Now the map $[A,tX]$ is invertible, since the map $tX$ is invertible by assumption. So $t[A,X]$ is invertible and $[A,X]$ is $T$\=/closed.

\smallskip
\noindent\ref{thm:monoidal-localization:3}
The first and second statement are proved in Lemma~\ref{lem:tensorlocalnew}. As a right adjoint the inclusion $\mc{V}^P\subset \mc{V}$ preserves finite limits and it preserves cotensors by~\ref{thm:monoidal-localization:2}. The endofunctor $P:\mc{V}\to \mc{V}$ is docile by Lemma~\ref{lem:Qisgood}. So the localization functor $P:\mc{V}\to \mc{V}^P$ preserves finite limits and compact cotensors and is therefore \mc{V}\=/left exact.

\smallskip
\noindent\ref{thm:monoidal-localization:4}
Lemma~\ref{lem:P-equivalences} allows to apply \cite[Proposition 4.1.7.4]{HA}: the functor $\otimes : \mc{V}\times \mc{V} \to \mc{V}$ induces a functor $\otimes^P : \mc{V}^P\times \mc{V}^P\to \mc{V}^P$ yielding a symmetric monoidal structure on $\mc{V}^P$  such that
\[
X\otimes_P Y=P(X)\otimes_P P(Y)=P(X\otimes Y)\,,
\]
for all  $X,Y$ in $\mc{V}^P$. Its unit is $\one_P=P(\one)$ and the symmetry is $\sigma_P(X,Y)=P(\sigma(X,Y)): P(X\otimes Y)\cong P(Y\otimes X)$. Clearly, the functor $P:\mc{V}\to \mc{V}^P$ is symmetric monoidal.

It remains to prove that this structure is closed. For every $X$ in \mc{V} and every $W$ in $\mc{V}^P$, the cotensor $[X,W]$ is in $ \mc{V}^P$ by~\ref{thm:monoidal-localization:2}.
Moreover, $[P(X\otimes Y), W]= [X\otimes Y, W]$ for every $X,Y$ in \mc{V} by Lemma~\ref{lem:tensorlocalnew}.
Thus, for all $X,Y,W$ in $\mc{V}^P$: 
\[
[X\otimes_P Y, W]= [X\otimes Y, W]= [Y,[X, W]]\,.
\]

\smallskip
\noindent\ref{thm:monoidal-localization:4}
As a reflective subcategory of \mc{V}, the category $\mc{V}^P$ is cocomplete and the localization functor $P:\mc{V}\to \mc{V}^P$ is cocontinous, since it is left adjoint to the inclusion functor $i: \mc{V}^P\to  \mc{V}$.
As an endofunctor, $P:\mc{V}\to \mc{V}$ is docile by Lemma~\ref{lem:Qisgood}.
It follows that filtered colimits are preserved by the inclusion $i:\mc{V}^{P}\subset \mc{V}$. Therefore, by Proposition~\ref{prop:confinedbasic2}, the localization functor $P:\mc{V}\to \mc{V}^P$ is confined.
(Beware that a compact object in $\mc{V}^P$ may not be compact in \mc{V}.)
In particular, $P(\comp(\mc{V})) \subset\comp(\mc{V}^P)$.

\smallskip
\noindent\ref{thm:monoidal-localization:5}
Since \mc{V} is $\omega$\=/presentable, for every object $X$ in \mc{V} we have $X=\colim_{\comp(\mc{V})/X} F(D)$ by Lemma~\ref{lem:objects-are-filtered-colimits} and the category $\comp(\mc{V})/X$ of compact object over $X$ is filtered. 
Suppose now that $X$ is $P$\=/closed. Then $X\cong PX\cong\colim_{\comp(\mc{V})/X} PF(D)$ is the corresponding colimit reflected into $P(\comp(\mc{V}))\subset \comp(\mc{V}^P)$. This implies that $P(\comp(\mc{V}))$ is dense in $\mc{V}^P$, that $\mc{V}^P$ is $\omega$\=/presentable and that every object in $\comp(\mc{V}^P)$ is a retract of an object in $P(\comp(\mc{V}))$.

It remains to show that the symmetric monoidal category $\mc{V}^P$ is confined. Since $\one$ is compact and $P$ is confined, the unit $\one_P=P(\one)$ is compact. 
Further:
\[
P(\comp(\mc{V}))\otimes_P P(\comp(\mc{V}))=P\bigl(\comp(\mc{V})\otimes \comp(\mc{V})\bigr)\subset P(\comp(\mc{V})) \subset \comp(\mc{V}^P)
\]
by~\ref{thm:monoidal-localization:4}. Since $P(\comp(\mc{V}))$ is dense in $\mc{V}^P$, 
it follows from Lemma~\ref{sufficientforconfinedtensor} that the symmetric monoidal closed category $\mc{V}^P$ is confined. 
\end{proof}

It follows that the localization $\mc{V}\to\mc{V}^P$ is a morphism of confined symmetric monoidal functors with the additional property that it preserves internal homs.

\subsection{The associated factorization system}
\label{sec:associated-fs}
The notion of a saturated class was defined in~\cite[Definition 5.5.5.1]{HTT}. See also \cite[Definition 3.1.12]{ABFJ:higher-sheaves} and \cite[Proposition 3.1.14]{ABFJ:higher-sheaves}.
The left class of a factorization system is always saturated. If \mc{V} is presentable and $\mc{R}=S^{\perp}$ is the class of maps right orthogonal to a set $S$ of morphisms of \mc{V}, then $\mc{L}={}^\perp(S^\perp)=:S\sat$ is the saturated closure of $S$ and $(\mc{L},\mc{R})$ form a factorization system in \mc{E}. See \eg \cite[Proposition 5.5.5.7]{HTT} or \cite[Proposition 3.1.18]{ABFJ:higher-sheaves}.

If \mc{V} has finite limits, then we define a {\it left exact modality}~\cite[Definition 4.1.1]{ABFJ:higher-sheaves} as a factorization system whose left class \mc{L} is the class of morphisms that is inverted by a left exact localization of \mc{V}. Equivalently, \mc{L} is a saturated class closed under finite limits~\cite[Lemma 4.1.2]{ABFJ:higher-sheaves}. We call such a class of morphisms a {\it congruence}~\cite[Definition 4.2.1]{ABFJ:higher-sheaves}.

\begin{proposition}\label{prop:enrichedsatgen}
Suppose that \mc{V} is a confined symmetric monoidal category and that $\mc{D}$ is a small dense subcategory of \mc{V}. 
Let $z:Z\to \one$ be a tidy map and $S:=\{z\otimes D :Z\otimes D \to D\,|\, D\in \mc{D}\}$.
Further, let $\mc{L}$ be the class of $P$\=/equivalences and $\mc{R}$ the class of $P$\=/closed maps. Then:
\begin{enumerate}[label={\rm (\roman*)}]
   \item\label{prop:enrichedsatgen:1}
The pair $(\mc{L}, \mc{R})$ is a left exact modality in \mc{V}, \ie \mc{L} is a congruence.
   \item\label{prop:enrichedsatgen:2}
$\mc{R}=S^{\perp}$ and $\mc{L}={}^\perp(S^\perp)=S\sat$.
\end{enumerate}
\end{proposition}

The point of this proposition is that the congruence \mc{L} is generated by $S$ as a saturated class.

\begin{proof}
For~\ref{prop:enrichedsatgen:1} note that the functor $P$ is a left exact reflector by Theorem~\ref{thm:monoidal-localization}. 
It then follows by~\cite[Proposition 4.1.6]{ABFJ:higher-sheaves} that $(\mc{L}, \mc{R})$ is a factorization system and therefore a left exact modality.

Let us prove~\ref{prop:enrichedsatgen:2}.
A map $f:X\to Y$  in \mc{V} is $P$\=/closed if and only if it is $T$\=/closed by Lemma~\ref{lem:Tn-local-Pn-local-mapsA}. By definition, $f:X\to Y$ is $T$\=/closed if and only if the following square is cartesian:
\begin{equation*} 
   \begin{tikzcd}[bo column sep=large]
   X \ar[r, "{[z, X]}"] \ar[d, "f"']   & {[Z,X]} \ar[d, "{ [Z,f] }"] \\
   Y \ar[r, "{[z, Y]}" ]   & {[Z,Y]}\,.
   \end{tikzcd}
\end{equation*} 
But, by Lemma~\ref{lem:detectinvertible}, this square is cartesian if and only if the square
\begin{equation*} 
   \begin{tikzcd}[bo column sep=large]
   \map(D,X) \ar[rrr, "{\map(D,{[z, X]})}"] \ar[d,"{\map(D,f)}"']   &&&  \map(D, {[Z,X]}) \ar[d, "{\map(D, {[Z,f]})}"]  \\
   \map(D,Y) \ar[rrr, "{\map(D,{[z, Y]})}" ]   &&& \map(D, {[Z,Y]})
   \end{tikzcd}
\end{equation*}
is cartesian, for every object $D$ in $\mc{D}$. This square is isomorphic to the square
\begin{equation*}
   \begin{tikzcd}[bo column sep=large]
   \map(D,X) \ar[rrr, "{\map(z\otimes D,X)}"]  \ar[d,"{\map(D,f)}"']     &&&  \map(Z\otimes D, X) \ar[d, "{\map(Z\otimes D, f)}"] \\
   \map(D,Y) \ar[rrr, "{\map(z\otimes D,Y)}" ]   &&& \map(Z\otimes D,Y)\,,
   \end{tikzcd}
\end{equation*} 
which is cartesian if and only if the maps $z\otimes D:Z\otimes D \to D$ are left orthogonal to the map $f:X\to Y$ for all $D$. In summary, a map $f$ is $P$\=/closed if and only it is right orthogonal to every map in $S$.

This shows that $\mc{R}=S^{\perp}$. But $\mc{L}=\mbox{}^{\perp}\mc{R}$, since the pair $(\mc{L},\mc{R})$ forms a factorization system by the first part of the proof. Thus, $\mc{L}=\mbox{}^{\perp}\mc{R}=\mbox{}^{\perp}(S^{\perp})=S\sat$. 
\end{proof}

\section{Intermezzo on the fiberwise join of maps} 
\label{Section:Intermezzo}

For our applications in the next section, we need to recall some facts about the pushout product and the join powers of maps in a category $\mc C$ with finite limits and universal finite colimits.

\subsection{Cocartesian gap maps and pushout products}

Consider the sets $\ul{n} =\{1,\hdots, n\}$ and the poset $[1]=(0<1)$.
An {\it $n$\=/cube} in the category \mc{C} is a functor $\chi:[1]^n\to \mc{C}$ or equivalently a functor $\chi:\mc{P}(\ul{n})\to \mc{C}$, where $\mc{P}(\ul{n})$ is the poset of subsets of $\ul{n}$.
For a cube $\chi:\mc P(\ul{n}) \to \mc{C}$ we call 
\[
\cogap(\chi):\colim_{U\subsetneq \ul{n}} \chi(U) \to \chi(\ul{n})
\]
its {\it cocartesian gap map}. 

\begin{definition}
\label{def:pp}
If \mc{C} admits finite products, then we define for two maps $f:A\to B$ and $g:C\to D$ in \mc{C} their {\it pushout product} $f\pp g$ to be the cocartesian gap map of the square
\[
\begin{tikzcd} 
A\times C \ar[rr,"{\id_A\times g}"] \ar[d,"{f\times\id_C}"'] && A\times D \ar[d,"{f\times\id_D}"]      \\
B\times C \ar[rr, "{\id_B\times g}"] && B\times D\,.
 \end{tikzcd}
\]
\end{definition}

\begin{remark}
\label{rem:pp}
The previous definition of the pushout product is relative the the cartesian product of the category $\mc C$. 
We could have used any monoidal structure on $\mc C$, but in this paper, we will only consider the pushout product with respect to cartesian structure.
\end{remark}

The {\it external cartesian product} $\chi\boxtimes \psi$ of two cubes $\chi:[1]^m\to \mc{C}$ and $\psi:[1]^n\to \mc{C}$ is the cube $\chi\boxtimes \psi:[1]^{m+n}\to \mc{C}$ defined by putting $(\chi\boxtimes \psi)(a,b)=\chi(a)\times \psi(b)$ for every $(a,b)$ in $[1]^m\times [1]^n$. 
It is easy to check that  $\cogap(\chi\boxtimes \psi)=\cogap(\chi)\pp \cogap(\psi)$. Then for any sequence $(f_1:X_1\to Y_1, \ldots, f_n:X_n\to Y_n)$ of maps, aka $1$\=/cubes, in $\mc{C}$, it follows that 
\begin{equation}\label{eq:fat-wedge}
   f_1\pp \cdots \pp f_n=\cogap (f_1\boxtimes \cdots \boxtimes f_n):{\rm FW}\to Y_1\times\hdots\times Y_n ,
\end{equation}
where FW is a relative fat wedge, because in the case of $X_1=\hdots= X_n=1$ it is really the fat wedge of $Y_1,\hdots,Y_n$. 

\subsection{The join product}

For two spaces $A$ and $B$, their join product $A\star B$ is defined as the codomain of the map $(A\to 1)\pp (B\to 1)=(A\star B\to 1)$. If $A$ and $B$ are finite spaces then $A\star B$ is finite as well.
The empty space $\emptyset$ is a unit for the the join product.

\begin{lemma}\label{lem:join-is-inf-sym-mon}
The category of finite spaces equipped with the join product is a symmetric monoidal ${(\infty,1)}$\=/category.
\end{lemma}

\begin{proof}

Example~\ref{exam:confined-sym-mon-cats}\ref{exam:confined-sym-mon-cats:6} 
explains that the Day convolution product yields a symmetric monoidal ${(\infty,1)}$\=/category. The pushout product on the category of morphisms $\mc{S}^{[1]}$ coming from the cartesian product on \mc{S} can be seen as a Day convolution where one employs the minimum on the poset $[1]$ as the monoidal product. Therefore the pushout product is a symmetric monoidal ${(\infty,1)}$\=/structure on $\mc{S}^{[1]}$. 

The inclusion of the full subcategory of morphisms of the form $A\to 1$ into the category of all morphisms is symmetric monoidal. The join product is by definition the restriction of the pushout product. Hence it is itself a symmetric monoidal ${(\infty,1)}$\=/structure.
Now one can restrict further from the category \mc{S} to $\Fin$ since the join of two finite spaces is again finite.
\end{proof}

For a space $K$ and a set $U$ let $K^U$ be the ordinary cotensor, in other words the $|U|$\=/fold cartesian power of $K$. 
\begin{lemma}\label{lem:calculate-zeta0}
For every object $K$ in $\mc{S}$ and every $n\geq 0$ we have
\[
K^{\star n}= \colim_{\emptyset \neq U\subset \ul{n}}  K^{U}\,.
\]
\end{lemma}

\begin{proof} 
The map $K^{\star n}\to 1\cong (K\to 1)^{\Box n}$ is the cocartesian gap map of the $n$\=/cube $\chi:=(K\to 1)^{\boxtimes n}$. By construction, $\chi(U)= K^{\complement{U}}$ for every subset $U\subsetneq \ul{n}$.
Hence 
\begin{equation*} 
K^{\star n}=\colim_{U\subsetneq \ul{n}}  K^{\complement{U}}=\colim_{\emptyset \neq U\subset \ul{n}}  K^{U}
\end{equation*}
for every $n\geq 0$. 
\end{proof}

\subsection{Fiberwise joins}
\label{sec:fiberwise-join}
For two maps $f:F\to B$ and $g:G\to B$ in \mc{S} the square  
\begin{equation*}
\begin{tikzcd} 
F\times_B G\ar[r,"{\pr_2}"] \ar[d,"{\pr_1}"']& G \ar[d,"{g}"]      \\
F  \ar[r, "{f}"] & B  
 \end{tikzcd}
\end{equation*}
is cartesian and we write $f\boxtimes_B g$ for it. 
We call its cocartesian gap map 
\[
f\star_B g=\cogap(f\boxtimes_B g): F\star_B G\to B
\]
the {\it fiberwise join} of $f$ and $g$.
More details are given in~\cite[Example 2.3.1\,(v)]{ABFJ:gbm2}.
Since colimits are universal in \mc{S}, the square
\[
\begin{tikzcd}
F\sqcup_{F\times_B G}G\ar[r] \ar[d,"{f\star_B g}"']& (F\times B) \sqcup_{(F\times G)} (B\times G) \ar[d,"{f\pp g}"]       \\
B  \ar[r, "{\Delta(B)}"] & B\times B 
\end{tikzcd}
\]
is cartesian where $\Delta(B)$ is the diagonal map of $B$.

An $n$\=/cube is {\it strongly cartesian} if every 2-dimensional face is a cartesian square.

\begin{lemma}\label{lem:fiberwise-join-is-base-change}
Let $(f_1,\ldots,f_n)$ be a sequence of maps $f_i:F_i\to B$ in \mc{S}. Then:
\begin{enumerate}[label={\rm (\roman*)}]
   \item\label{lem:fiberwise-join-is-base-change:1} 
   The $n$\=/cube $f_1\boxtimes \cdots \boxtimes f_n$ is strongly cartesian and its base change along the diagonal map $B\to B^n$  is the strongly cartesian $n$\=/cube $f_1\boxtimes_B \cdots \boxtimes_B f_n$.
   \item\label{lem:fiberwise-join-is-base-change:2}
The pushout product $f_1\pp \cdots \pp f_n$ is the cocartesian gap map of the $n$\=/cube $f_1\boxtimes \cdots \boxtimes f_n$.
   \item\label{lem:fiberwise-join-is-base-change:3}
The fiberwise join product $f_1\star_B\cdots \star_B f_n$ is the cocartesian gap map of the $n$\=/cube $f_1\boxtimes_B \cdots \boxtimes_B f_n$.
   \item\label{lem:fiberwise-join-is-base-change:4}
   The map $f_1\star_B\cdots \star_B f_n$ is the base change of $f_1\pp \cdots \pp f_n$ along the diagonal $B\to B^n$.
\end{enumerate}
\end{lemma}

\begin{proof}
Statements~\ref{lem:fiberwise-join-is-base-change:2},~\ref{lem:fiberwise-join-is-base-change:2} and~\ref{lem:fiberwise-join-is-base-change:3} follow by direct inspection.
Statement~\ref{lem:fiberwise-join-is-base-change:4} follows from the universality of colimits in \mc{S}.
\end{proof}

\section{Goodwillie calculus revisited} 
\label{section:Goodcalcrev}
We denote by $\Fin\subset\mc{S}$ the full subcategory of finite spaces. It is the free finitely cocomplete category on one generator. 
The category $\Fin$ together with the join product $-\star-$ becomes a symmetric monoidal structure as formulated in Lemma~\ref{lem:join-is-inf-sym-mon}.
The Day convolution product makes $\Fun(\Fin, \mc{S})$ into a symmetrice monoidal category, which is confined by Example~\ref{exam:confined-sym-mon-cats}\,\ref{exam:confined-sym-mon-cats:6}.
By slight abuse of notation we write $\Id$ for the inlcusion functor $\Fin\subset\mc{S}$.
For each $n\ge 0$ we choose $z_{n+1}:\Id^{\star n+1}\to 1=(\Id\to 1)^{\square n+1}$ as our map $z:Z\to \one$ to start the machinery of Section~\ref{sec:Quahog}. The fact that $z_{n+1}$ is tidy, proved in Lemma~\ref{r-star-is-On-hence-Pn-equivalence137}, yields all the good properties of the reflector $P_n$.

\subsection{The reflector \texorpdfstring{$P_n$}{Pn}}

\begin{definition}
Define $T_n$ as an endofunctor of $\Fun(\Fin, \mc{S})$ by cotensoring $T_n(F):=[\Id^{\star n+1} ,F]$. 
\end{definition}

\begin{lemma}\label{lem:calculate-zeta2}
For every $n\geq 0$ we have  
\begin{equation*}
  \Id^{\star n+1}= \colim_{\emptyset \neq U\subset \ul{n+1}} \map(U, -)
\end{equation*}
In particular, the ${(n+1)}$\=/fold join power $\Id^{\star n+1}$ is a finitely presentable object in $\Fun(\Fin, \mc{S})$.
\end{lemma}

\begin{proof} 
This follows from Lemma~\ref{lem:calculate-zeta0}.
\end{proof}

Since the internal hom here is taken with respect to the Day convolution product coming from the join on $\Fin$, and is not the usual internal hom on $\Fun(\Fin, \mc{S})$. For every $F:\Fin\to \mc{S}$ we have 
\begin{equation*} 
  T_nF=[\Id^{\star n+1}, F]= \lim_{\emptyset\neq U\subset \ul{n+1}} F(U\star -)
\end{equation*}
and the canonical map $z_{n+1}:\Id^{\star n+1}\to 1$ yields a natural transformation
\[
t_n(F)=[z_{n+1},F]: F \to T_n(F)\,.
\]
As in Construction~\ref{def:PZ}  let $P_n$ be the endofunctor of $\Fun(\Fin, \mc{S})$ defined by 
\begin{equation}\label{def:Good-Pn} 
   P_n:=\colim \bigl( \Id \xto{t_n} T_n\xto{t_nT_n} T_n^2\xto{t_nT_n^2} T_n^3\xto{t_nT_n^3} T_n^4\to \hdots \bigr) 
\end{equation} 
and let $p_n:\Id \to P_n$ be the canonical map. 

\subsection{The map \texorpdfstring{$z_{n+1}$}{zn+1} is tidy}
\label{subsection:unpointed-Goodwillie}

This will be proved in Lemma~\ref{r-star-is-On-hence-Pn-equivalence137} and follows from connectivity estimates. We define a map to be $n$-connected if it is in the acyclic class generated by the map $S^{n+1}\to 1$. In the category \mc{S} this is equivalent to requiring that every homotopy fiber is $n$-connected as a space, see~\cite[Section 3.3]{ABFJ:gbm}. 
{\it This differs from the classical definition in homotopy theory!}

\begin{definition}[{\cite[Definition 1.2]{Good03}}]
\label{def:On}
A map $\alpha:F\to G$ in $\Fun(\Fin, \mc{S})$ is said to satisfy condition $O_n(c,\kappa)$ for $c$ in $\mb{Z}$ and $\kappa \geq -2$  if the connectivity of the map $\alpha(K):F(K)\to G(K)$ is $\geq (n+1)k-c$ for every $K$ in $\Fin$ of connectivity $k \geq \kappa$.
\end{definition}

\begin{proposition}[{\cite[Proposition 1.6]{Good03}}]
\label{prop:Pn-and-On}
If a map $\alpha:F\to G$ in $\Fun(\Fin, \mc{S})$ satisfies condition $O_n(c, \kappa)$ for some $c$ and $\kappa$, then $P_n(\alpha):P_nF\to P_nG$ is invertible. 
\end{proposition}

Although this was stated in~\cite{Good03}, the proof relies only on connectivity estimates that Goodwillie had developped earlier. Note that if $\alpha:\mc{X}\to\mc{Y}$ is a map of $n$\=/cubes, then it can be viewed as an ${(n+1)}$\=/cube which we denote $[\alpha]$.
\begin{lemma}[{\cite[Proposition 1.6]{Good92}}]  
\label{lem:cubes-cartesian17}
Let $\alpha:\mc{X}\to\mc{Y}$ be a map of $n$\=/cubes in $\mc{S}$.
\begin{enumerate}[label={\rm (\roman*)}]
  \item
If the ${(n+1)}$\=/cube $[\alpha]$ is $k$\=/cartesian and $\mc{Y}$ is $k$\=/cartesian, then $\mc{X}$ is $k$\=/cartesian.
  \item
If $\mc{X}$ is $k$\=/cartesian and $\mc{Y}$ is ${(k+1)}$\=/cartesian, then $[\alpha]$ is $k$\=/cartesian.
\end{enumerate}
\end{lemma}

The next statement is Goodwillie's \cite[Theorem 1.20]{Good92} simplified to the case $T=\ul{1}$.
\begin{lemma} \label{lem:cubes-cartesian2}
Let $\alpha:\mc{X}\to\mc{Y}$ be a morphism of $n$\=/cubes in \mc{S}. Suppose that the ${(n+1)}$\=/cube $[\alpha]$ is $k$\=/cartesian and that the map $\alpha(U)$ is ${(k+|U|-1)}$\=/connected for every non-empty subset $U\subset \ul{n}$. Then the map 
$\alpha(\emptyset)$ is $k$\=/connected.
\end{lemma}

\begin{lemma}\label{r-star-is-On-hence-Pn-equivalence137}
The map $z_{n+1}:\Id^{\star n+1} \to 1$ is tidy. 
\end{lemma}

\begin{proof} 
When we evaluate the map $z_{n+1}: \Id^{\star n+1}\to 1$ at a $k$\=/connected finite space $K$, the resulting map $K^{\star n+1}\to 1$ is $((n+1)k+2n)$\=/connected. Hence the map $z_{n+1}$ satisfies condition $O_n(-2n,-1)$. By Proposition~\ref{prop:Pn-and-On} the map $P_n(z_{n+1})$ is invertible.
\end{proof}

\begin{theorem}[Goodwillie] \label{thm:Pn-reflects-onto-Tn} 
The functor $P_n$ defined in~\eqref{def:Good-Pn} is a left exact reflector onto the subcategory of $P_n$\=/closed objects of $\Fun(\Fin,\mc{S})$. The $P_n$\=/closed objects are exactly the $n$\=/excisive functors. 
\end{theorem}
 
\begin{proof} 
Lemma~\ref{r-star-is-On-hence-Pn-equivalence137} states that the map $z_{n+1}$ is tidy, so Theorem~\ref{thm:monoidal-localization} applies. 
It remains to identify the $P_n$\=/closed objects with the $n$\=/excisive functors. By Theorem~\ref{thm:monoidal-localization}\,\ref{thm:monoidal-localization:1} a functor in $\Fun(\Fin,\mc{S})$ is $P_n$\=/closed if and only if it is $T_n$\=/closed.
The fact that $T_n$\=/closed objects coincide with $n$\=/excisive functors is proved in~\cite[Theorem 4.4.5]{ABFJ:product} based on the observation in~\cite[Lemma 4.4.2]{ABFJ:product} that all strongly cocartesian cubes are obtained from free cocartesian cubes by cobase change. 

The original way to show that $T_n$\=/closed objects and $n$\=/excisive functors are the same, is Goodwillie's~\cite[Lemma 1.9]{Good03}, see also~\cite{Rezk13}, based on clever manipulations with cubical diagrams.
\end{proof}

Using Lemma~\ref{lem:calculate-zeta2}, we found
\begin{align*}
   \map(K,-) \otimes z_{n+1}  &=\map(K,-)\otimes ( \Id^{\star n+1} \to 1) \\
            &=\left( \map(K,-)\otimes\colim_{\emptyset \neq U\subset \ul{n+1}} \map(U, -) \to \map(K,-) \right) \\               
            &=\left( \colim_{\emptyset \neq U\subset \ul{n+1}} \map(K\star U, -) \to\map(K,-) \right)\,.
\end{align*} 
By Proposition~\ref{prop:enrichedsatgen},
the right orthogonal class of these maps is the class of $T_n$\=/closed maps. 
The local objects are then the $n$\=/excisive functors. This calculation has a nice parallel in orthogonal calculus in~\eqref{eq:nice-parallel}.

\subsection{\texorpdfstring{$\Fun(\Finp,\mc{S})$}{Semi-pointed Goodwillie calculus}}

Let $\Finp$ denote the category of pointed finite spaces together with the smash product.
By Examples~\ref{exam:confined-sym-mon-cats}\,\ref{exam:confined-sym-mon-cats:6}, the category $\Fun(\Finp,\mc{S})$ is a confined symmetric monoidal category for the Day convolution product.

We write $\map_{\ast}(A, -)$ for the functor represented by the pointed space $A$. The functor $\Id_\circ:\Finp\to\mc{S}$ forgetting the base point is represented by the pointed space $S^0$.
Since $S^0$ is the unit of the smash product, $\Id_\circ$ is the unit of the Day convolution product.
The functor $\map_\ast(1,-) = 1$ is the terminal functor $\Finp\to\mc{S}$.
The map $S^0\to 1$ induces a natural transformation 
\[
z_0:1=\map_\ast(1,-)\to\map_{\ast}(S^0, -)=\Id_\circ\,.
\]
For each $n\ge 0$ the square
\begin{equation}\label{eq:def-zn+1}
\begin{tikzcd}
 \Gamma_{n+1}(\Id_\circ)  \ar[rr] 
 \ar[d,"{z_{n+1}:=z_0^{\star n+1}}"'] && {\rm FW}_{n+1}(\Id_\circ) \ar[d,"{z_0^{\square n+1}}"] \\
     \Id_\circ \ar[rr,"\Delta"] && \Id_\circ^{n+1}
\end{tikzcd}
\end{equation}
is cartesian square by Lemma~\ref{lem:fiberwise-join-is-base-change}\,\ref{lem:fiberwise-join-is-base-change:4}. 
We will use this map $z_{n+1}$ to start the machinery of Section~\ref{sec:Quahog}. The tidyness of $z_{n+1}$ is proved in Lemma~\ref{tidy-pointed}.

\begin{lemma}
\label{lem:pointed:calculate-zeta2}
For every $n\geq 0$, the domain of the map $z_{n+1}:(1\to \Id_\circ)^{\star n+1}$ is the functor
\begin{equation*}
   \Gamma_{n+1}(\Id_\circ) = \colim_{\emptyset \neq U\subset \ul{n+1}} \map_{\ast}(\Sigma U,-) \ .
\end{equation*}
In particular, it is a finitely presentable object in $\Fun(\Finp,\mc{S})$.
\end{lemma}

\begin{proof} 
By~\eqref{eq:fat-wedge} the codomain of the iterated pushout product $(1\to \Id_\circ)^{\pp n+1}$ is the fat wedge:
\[
  {\rm FW}_{n+1}(\Id_\circ)= \colim_{\emptyset \neq U\subset \ul{n+1}} \Id_\circ^{\complement U}\ =\ \colim_{\emptyset \neq U\subset \ul{n+1}} \map_{\ast}((\complement U)_+,\Id_\circ)\ .
\]
The codomain of the iterated join $z_{n+1}=(1\to \Id_\circ)^{\star n+1}$ is by Diagram~\eqref{eq:def-zn+1} the base change of this colimit along the diagonal $\Id_\circ \to (\Id_\circ)^{n+1}$. To compute the pullback observe: for $U\subset\ul{n+1}$ the pullback of $(\Id_\circ)^{\ul{n+1} - U}\to (\Id_\circ)^{n+1}\leftarrow \Id_\circ$ is the Yoneda image of the span of pointed spaces $(\ul{n+1} \setminus U)_+\leftarrow \ul{n+1}_+ \to 1_+$, whose pushout is $\Sigma U$. 
This is the unreduced suspension of the unpointed set $U$, pointed at the north pole say.
So the codomain of $z_{n+1}$ is $\Id_\circ^{\Sigma U}$.
Then, by universality of colimits, the codomain of $z_{n+1}$ is $\colim_{\emptyset \neq U\subset \ul{n+1}} \Id_\circ^{\Sigma U}$.
This is a finite colimit of representable functors and proves the last statement.
\end{proof}

As in Construction~\ref{def:PZ} one gets $T_n$ as an endofunctor of $\Fun(\Fin, \mc{S})$ by cotensoring:
\begin{align*}
   T_n(F) &:=\left[\colim_{\emptyset \neq U\subset \ul{n+1}} \map_{\ast}(\Sigma U,-), F\right]\\
         & = \lim_{\emptyset \neq U\subset \ul{n+1}} F(\Sigma U\wedge -) \\
         & = \lim_{\emptyset \neq U\subset \ul{n+1}} F(U\star -). 
\end{align*}
using the natural isomorphism $A\wedge \Sigma U = A\star U$ for a pointed space $A$ and an unpointed space $U$. Again Construction~\ref{def:PZ} yields $P_n$ which takes the same form as in the previous section.

\begin{lemma}
\label{tidy-pointed}
The map $z_{n+1}:(1\to \Id_\circ)^{\star n+1}$ is tidy. 
\end{lemma}

\begin{proof} 
Proposition~\ref{prop:Pn-and-On} applies here as well since the $P_n$ obtained from $z_{n+1}$ here is the same as in the previous section. Let $K$ be a $k$\=/connected space for $k\geq -1$.
The map $1\to K$ is ${(k-1)}$\=/connected since its fiber is $\Omega K$.
So the map $(1\to K)^{\star n+1}$ is $\ell$\=/connected for 
\[
\ell \ =\ (n+1)(k-1)+2n \ =\ (n+1)k + n-1
\]
Hence the map $z_{n+1}$ satisfies condition $O_n(-n+1,-1)$ and $P_n(z_{n+1})$ is invertible.
\end{proof}

The result is the analogous of Theorem~\ref{thm:Pn-reflects-onto-Tn}: 
\begin{theorem}[Goodwillie]
\label{thm:pointed-Pn-reflects-onto-Tn} 
Goodwillie's functor $P_n$ is a left exact reflector onto the subcategory of $n$\=/excisive functors in $\Fun(\Finp,\mc{S})$. 
\end{theorem}

\subsection{\texorpdfstring{$\Fun_\ast(\Finp,\mc{S}_{\ast})$}{Pointed Goodwillie calculus}}

We denote by $\mc{S}_{\ast}$ the category of pointed spaces.
We write $\map_{\ast}(-,-)$ for the pointed mapping space.
Let $\Fun_\ast(\Finp,\mc{S}_{\ast})$ be the category of $\mc{S}_{\ast}$\=/enriched functors from finite pointed spaces to pointed spaces. This category is itself $\mc{S}_\ast$\=/enriched and we write ${\rm nat}_{\ast}(-,-)$ for this pointed space of $\mc{S}_\ast$\=/enriched natural transformations.
If we equip both categories $\Finp$ and $\mc{S}_\ast$ with the smash product, $\Fun_\ast(\Finp,\mc{S}_{\ast})$ with the associated Day convolution is a symmetric monoidal closed category. Its unit is the functor $\Id_\ast:=\map_{\ast}(S^0, -)=\one$.

In Subsection~\ref{subsec:csmc-cat} the concept of confined symmetric monoidal category was defined with respect to unpointed mapping spaces. To prove that $\Fun_\ast(\Finp,\mc{S}_{\ast})$ is confined, we consider it as being enriched over \mc{S}. This is possible since forgetting the base point $u:\mc{S}_\ast\to\mc{S}$ is a lax symmetric monoidal and therefore yields an enrichment of $\Fun_{\ast}(\Finp,\mc{S}_{\ast})$ over \mc{S}. It is given by $u({\rm nat}_{\ast}(-,-))$.

\begin{lemma}
The category $\Fun_\ast(\Finp,\mc{S}_{\ast})$ is a confined symmetric monoidal category.
\end{lemma}

\begin{proof}
The enriched Yoneda functor $\Finp^{\op}\to\Fun_\ast(\Finp,\mc{S}_{\ast})$ induces a functor
\[
N: \Fun_\ast(\Finp,\mc{S}_{\ast}) \to \Fun(\Finp,\mc{S})\,.
\]
A functor $F$ is reduced if $F(1)=1$ and we denote the full subcategory of reduced functors by a superscript $(-)^{\rm red}$.
Now $N$ can be factored in the following way 
\begin{align*}
  \Fun_\ast(\Finp,\mc{S}_{\ast}) = \Fun(\Finp,\mc{S}_{\ast})^{\rm red} = \Fun(\Finp,\mc{S})^{\rm red}\subset\Fun(\Finp,\mc{S}),
\end{align*}
where we have first two equivalences of categories and then the inclusion of a full subcategory. This inclusion preserves filtered colimits. Hence $N$ preserves filtered colimits. Therefore every representable functor $\map_{\ast}(A,-)$, $A$ pointed, is compact in $\Fun_\ast(\Finp,\mc{S}_{\ast})$.
Now the result follows from Lemma~\ref{sufficientforconfinedtensor}.
\end{proof}

Notice that the maps $z_{n+1}:(1\to \Id_\ast)^{\star n+1}$ of Lemma~\ref{lem:pointed:calculate-zeta2} in $\Fun(\Fin_\ast, \mc S)$ belong the subcategory
$\Fun_\ast(\Fin_\ast, \mc S_\ast)$.
The analogue of the cartesian square~\eqref{eq:def-zn+1} coming from Lemma~\ref{lem:fiberwise-join-is-base-change}\,\ref{lem:fiberwise-join-is-base-change:4} implies that there is an isomorphism
\[
z_{n+1}:(1\to \Id_\ast)^{\star n+1}=\left( \colim_{\emptyset\neq U\subset\ul{n+1}}\map_\ast(\Sigma U,-)\to\Id_\ast  \right)
\]
of maps where $\Sigma$ is the unreduced suspension. Further for any finite pointed space $K$:
\[
\map_\ast(K,-)\otimes z_{n+1} = \left( \colim_{\emptyset\neq U\subset\ul{n+1}}\map_\ast(K\star U,-)\to\map_\ast(K,-)   \right)\,.
\]
It is now clear that the endofunctors $T$ and $P$ of $\Fun(\Finp,\mc{S}_{\ast})$  associated to $z_{n+1}$ described in Construction~\ref{def:PZ} are the functors $T_n$ and $P_n$ constructed by Goodwillie in~\cite{Good03}. 
Now Goodwillie's connectivity estimate in Proposition~\ref{prop:Pn-and-On} can be used to prove the following:

\begin{theorem}\label{thm:very-pointed-Goodwillie} 
The maps $z_{n+1}:(1\to \Id_\ast)^{\star n+1}$ are tidy. Theorem~\ref{thm:monoidal-localization} applies and $P=P_n$ is a confined symmetric monoidal left exact localization. The reflector $P_n$ is Goodwillie's $P_n$.
\end{theorem}

\section{Orthogonal calculus revisited}
\label{sec:orthogonal-calculus}

The construction of the orthogonal tower given in Subsection~\ref{subsec:orth-con-est} is really the one from~\cite{Weiss98}, but we are recasting it in the framework of confined symmetric monoidal categories from Section~\ref{sec:Quahog}. Our main motivation though is to prove that orthogonal calculus is a special case of a completion tower developped in~\cite{ABFJ:product}. This is done in Subsection~\ref{subsec:ortho-tower-is-completion-tower}. As a bonus we obtain Blakers-Massey theorems for orthogonal calculus in Subsection~\ref{subsec:blakers-massey}.  

Let us point out that the symmetric monoidal structures appearing in orthogonal calculus were already studied by Hendrian~\cite{hendrian:monoidal-orth-calc}.

A warning on notation: we are using a uniform notation to highlight the parallels between the constructions of orthogonal calculus and Goodwillie calculus, but this conflicts with Weiss' choice who denotes our functor $T_n$ by $\tau_n$, and our $P_n$ by $T_n$. 

Let $\mc{J}$ be the category of finite dimensional Euclidean vector spaces. The space of morphisms in \mc{J} is the Stiefel manifold of linear isometries from $U$ to $V$. We are going to denote this by $\mc{J}(U,V)=\st(U,V)$ following Weiss. 
The orthogonal sum equips $\mc{J}$ with a symmetric monoidal structure. This yields an ${(\infty,1)}$\=/symmetric monoidal category (the interested reader can consult~\cite[Proposition/Definition 4.1.1.4]{hendrian:monoidal-orth-calc}).
The category $\Fun(\mc{J},\mc{S})$ can be given the corresponding Day convolution product with the terminal functor $\st(0,-)=1$ as unit. This symmetric monoidal closed category is confined by Examples~\ref{exam:confined-sym-mon-cats}\,\ref{exam:confined-sym-mon-cats:6}.

In \mc{J} there is the inclusion $i_k:\mb{R}^k\to\mb{R}^{k+1}$ using zero for the last coordinate. 
We denote $j_k:\st(\mb{R}^{k+1},-)\to\st(\mb{R}^{k},-)$ the induced maps on representable functors. 
We put $\sph(-)=\st(\mb{R},-):\mc{J} \to \mc{S}$ since this functor sends a vector space to its unit sphere. We feed the map
\[
j_0^{\Box n+1}:\sph(-)^{\star n+1}\to 1
\]
into the setup of Section~\ref{sec:Quahog}. This map is tidy by Proposition~\ref{prop:Pn-zetan-WeissA}. This relies on connectivity estimates proved by Weiss~\cite{Weiss98}.
As prescribed in Construction~\ref{def:PZ}, the natural transformation $t_n(F)=[z_{n+1},F]: F \to T_n(F)$ gives rise to the endofunctor 
\begin{equation*}
   P_n:=\colim \bigl( \Id \xto{t_n} T_n\xto{t_nT_n} T_n^2\xto{t_nT_n^2} T_n^3\xto{t_nT_n^3} T_n^4\to \hdots \bigr) 
\end{equation*} 
of $\Fun(\mc{J},\mc{S})$ with $p_n:\Id \to P_n$ the canonical map to the colimit.

\subsection{Stiefel combinatorics} 
\label{subsec:constructing-orth-tower}

\begin{lemma}\label{lem:product-fin-pres}
For every $k\ge 0$ there is a pushout
\[
\begin{tikzcd} 
   \st(\mb{R}^{k+1},V)\times \mathbb{S}^{k-1} \ar[rr,"{j_k\times\id}"] \ar[d,"{\pr_1}"']  && \st(\mb{R}^{k},V)\times \mathbb{S}^{k-1}  \ar[d] \\
   \st(\mb{R}^{k+1},V) \ar[rr] && \st(\mb{R}^k,-)\times\st(\mb{R},-)
\end{tikzcd}
\]
in $\Fun(\mc{J},\mc{S})$. In particular, the functor $\st(\mb{R}^k,-)\times\st(\mb{R},-)$ is finitely presentable.
\end{lemma}

\begin{proof}
In this proof we are working in the $1$\=/category of topological spaces and continuous maps.
Let us evaluate on a vector space $V$ of dimension $\ge k$, since otherwise the values are empty. 
Consider the pushout:
\[
\begin{tikzcd} 
   \st(\mb{R}^{k+1},V)\times \mathbb{S}^{k-1} \ar[rr,"{j_k\times\id}"] \ar[d,"{\id\times{\rm incl}}"']  && \st(\mb{R}^{k},V)\times \mathbb{S}^{k-1}  \ar[d] \\
   \st(\mb{R}^{k+1},V)\times \mb{D}^{k} \ar[rr] && Q 
\end{tikzcd}
\]
A linear isometry $f$ in $\st(\mb{R}^m, V)$ can be represented as an orthonormal frame $(f_1,\hdots,f_m)$. Therefore, set-theoretically, $Q$ is defined as the quotient of
\[
\bigl( \st(\mb{R}^{k+1},V) \times \mb{D}^{k}\bigr) \sqcup \bigl( \st(\mb{R}^{k},V)\times \mathbb{S}^{k-1} \bigr)
\]
by the relation
\[
\bigl((f_1,\hdots, f_k, f_{k+1}), (t_1,\hdots,t_k)\bigr) \sim \bigl((g_1,\hdots, g_k), (t_1,\hdots,t_k)\bigr)
\]
if and only if $f_i=g_i$ for all $1\le i\le k$ and $\fac{(t_1,\hdots,t_k)}=1$. Since the left vertical map is a cofibration, this square is a homotopy pushout.
We want to show that $Q$ is homeomorphic to $\st(\mb{R}^k,-)\times\st(\mb{R},-)$. 

Let us describe a map $q:Q\to \st(\mb{R}^k,V)\times\st(\mb{R},V)$. To construct the map we use the fact that $Q$ is a pushout. 
First let
\[
q_1: \st(\mb{R}^{k},V)\times \mathbb{S}^{k-1} \to \st(\mb{R}^k,V)\times\st(\mb{R},V)
\]
be given by
\[
(f_1,\hdots, f_k; t)\mapsto (f_1,\hdots, f_k\ ;\ \sum_{i=1}^{k}t_if_i )\,.
\]
Note that $\fac{\sum_{i=1}^{k}t_if_i}=\fac{t}=1$, so $\sum_{i=1}^{k}t_if_i$ is a vector on the boundary of the unit disc of $U={\rm span}(f_1,\hdots,f_k)$. Therefore the map $q_1$ is well-defined. Note also that $U$ varies continuously with $(f_1,\hdots,f_k)$, so $q_1$ is continuous.

On the part of $Q$ involving $\st(\mb{R}^{k+1},-)\times\mb{D}^k$ the idea is to use $t$ in $\mb{D}^k$ in the fiber over $(f_1,\hdots,f_k,f_{k+1})$, to tilt $f_{k+1}$ in such a way that its orthogonal projection into the span $U$ of $(f_1,\hdots, f_k)$ becomes $t$. The map
\[
q_2: \st(\mb{R}^{k+1},V)\times \mb{D}^{k} \to \st(\mb{R}^k,V)\times\st(\mb{R},V)
\]
is defined as follows: 
\[
(f_1,\hdots, f_k, f_{k+1}\ ;\ t)\mapsto (f_1,\hdots, f_k\ ;\ t+(1-\fac{t}^2)^{1/2}\,f_{k+1})
\]
with $t:=\sum_{i=1}^{k}t_if_i$ in $U$. 
Observe that
\begin{align*}   
   \fac{t+ \bigl(1-\fac{t}^2\bigr)^{1/2}\,f_{k+1}} =1 
\end{align*}
since $t\cdot f_{k+1}=0$ and $\fac{f_{k+1}}=1$. So $q_2$ is well-defined and continuous.

To obtain a map on the pushout $Q$ we need to check that $q_1(j_k\times\id)=(\id\times{\rm incl})q_2$. This is true because of the fact that, if $\fac{t}=1$, then $t+(1-\fac{t}^2)^{1/2}f_{k+1}=t=\sum_{i=1}^{k}t_if_i$.

Now let us define a map in the other direction.
An element in $\st(\mb{R}^k,V)\times\st(\mb{R},V)$ is represented by an orthonormal frame $(f_1,\hdots,f_k)$ and a single unit vector $e$ in $V$. Let $U$ be the span of the frame. If $p$ denotes the orthogonal projection of $e$ into $U$, then $\fac{p}\le\fac{e}=1$, so $p$ is in the unit disc of $U$ using the basis $(f_1,\hdots,f_k)$. Explicitly, $p=\sum_{i}(e\cdot f_i)f_i$ where the dot denotes the scalar product. Note $p$ in $ \mathbb{S}^{k-1}$ if and only if $e=p$ in $ U$. Now the map
\[
s:\st(\mb{R}^k,V)\times\st(\mb{R},V)\to Q
\]
is defined by applying the Gram-Schmidt process whenever possible. So let $s$ be given by
\[   \bigl((f_1,\hdots,f_k),e\bigr)\mapsto \left\{
          \begin{array}{cl}
            (f_1,\hdots,f_k, (e-p)/\fac{e-p}\ ;\ p) &\text{ if } e\notin U , \\ 
                           & \\
            (f_1,\hdots,f_k\ ;\ p)                         &\text{ if } e\in U .
          \end{array}\right. 
\]
Note that in the case $e\notin U$ we have $e\neq p$ and $(f_1,\hdots,f_k, (e-p)/\fac{e-p})$ is an orthonormal ${(k+1)}$\=/frame. So the image of $s$ lies in $\st(\mb{R}^{k+1},V)\times (\mb{D}^k - \mathbb{S}^{k-1})\subset Q$. In the case $e$ in $ U$ the image of $s$ lies in $\st(\mb{R}^k,V)\times\mathbb{S}^{k-1}\subset Q$. The map $s$ is continuous by the glueing that occurs in $Q$.

Now it is elementary to check that $q$ and $s$ are mutually inverse. 
\end{proof}

\begin{remark}
The pushout of Lemma~\ref{lem:product-fin-pres} says that the projection $\st(\mb{R}^k,-)\times\st(\mb{R},-)\to \st(\mb{R}^k,-)$ is the pushout product $j_k\Box (\mathbb{S}^{k-1}\to 1)$.
\end{remark}

\begin{lemma}\label{lem:join-with-comp-is-comp}
For every compact $F$ in $\Fun(\mc{J},\mc{S})$ the functor $\sph(-)\star F$ is compact.
\end{lemma}

\begin{proof}
Because $F$ is compact, it is the retract of a finitely presentable functor $G$. Thus it suffices to show that the join with $G$ is compact. Since $G$ is of the form $\colim_{c\in C}\st(\mb{R}^{k(c)},-)$ with $C$ a finite category, the join $\sph(-)\star G$ is the pushout of the diagram:
\[
\sph(-)\leftarrow \colim_{C} \left(\sph(-)\times\st(\mb{R}^{k(c)},-)\right) \to \colim_{C}\st(\mb{R}^{k(c)},-)\,.
\]
The two ends of the pushout are finitely presentable, and the center is a finite colimit of finitely presentable functors by Lemma~\ref{lem:product-fin-pres}. The claim follows.
\end{proof}

\begin{corollary}\label{cor:S-tothe-n+1-is-compact}
The functor $\sph(-)^{\star n+1}$ is compact. 
\end{corollary}

\begin{proof}
This follows inductively from Lemma~\ref{lem:join-with-comp-is-comp}.
\end{proof}

\begin{lemma}[Weiss] \label{Zn-finite}
In $\Fun(\mc{J},\mc{S})$ there is for all $n,k\geq 0$ a natural isomorphism 
\begin{equation*}
   \underbrace{\ j_k\star_B\hdots\star_B j_k\, }_{n+1} =\left(\colim_{0\neq U \subset \mb{R}^{n+1}} \st(\mb{R}^k\oplus U,-)\to B \right) 
\end{equation*}
of maps, where $B:=\st(\mb{R}^{k},-)$ and $\star_B$ is the fiberwise join $($see \emph{Subsection~\ref{sec:fiberwise-join}}$)$.
\end{lemma}

\begin{proof}
Using Weiss' notation in \cite[Proposition 4.2]{Weiss95} there is a natural isomorphism over $B:=\st(\mb{R}^{k},-)$
\[
\colim_{0\neq U \subset \mb{R}^{n+1}} \st(\mb{R}^k\oplus U,-) \xto{\cong} S\gamma_{n+1}(\mb{R}^k, -)\,,
\]
where the right hand side is the total space of the unit sphere bundle over $B$ obtained from the vector bundle $\gamma_{n+1}(\mb{R}^k, -)$ whose fiber at $f$ in $B$ is $\mb{R}^{n+1}\otimes\ {\rm coker}f$. 
Then in the proof of \cite[Proposition 5.4]{Weiss95} the author notes that $\gamma_{n+1}(\mb{R}^k,-)$ is in fact isomorphic as a vector bundle over $B$ to the Whitney sum $\bigoplus_{n+1} \gamma_{1}(\mb{R}^k,-)$. 
For the corresponding unit sphere bundle, one obtains an identification of the structure map $S\gamma_{1}(\mb{R}^k,-)\to B$ with the map $j_k:\st(\mb{R}^{k+1},-)\to\st(\mb{R}^k,-)$. For the unit sphere bundle the Whitney sum translates into a fiberwise join
\[
\bigl(\, S\gamma_{n+1}(\mb{R}^k,-)\to B\,\bigr)\cong \underbrace{\ j_k\star_B\hdots\star_B j_k\, }_{n+1}\,,
\]
proving the claim.
\end{proof}

For $k=0$ the target space is $B=\st(0,-)=1$, so the map in Lemma~\ref{Zn-finite} reduces to 
\begin{align*} 
   \left(\colim_{0\neq U \subset \mb{R}^{n+1}} \st(U,-)\to 1 \right)= \underbrace{\ j_0\star_B\hdots\star_B j_0\, }_{n+1}=j_0^{\Box n+1} =    z_{n+1}\,.
\end{align*}
Since the internal hom here is taken with respect to the Day convolution product coming from the orthogonal sum on $\mc{J}$, for every $F:\mc{J}\to \mc{S}$ we have 
\begin{equation*} 
   T_nF= \left[\sph(-)^{\star n+1}, F\right]=\left[\colim_{0\neq U \subset \mb{R}^{n+1}} \st(\mb{R}^k\oplus U,-),F\right]=\lim_{0\neq U \subset \mb{R}^{n+1}} F(U\oplus -)\,.
\end{equation*}

\subsection{Weiss' connectivity estimates}
\label{subsec:orth-con-est}

\begin{lemma}[{\cite[Lemma e.3]{Weiss98}}]\label{lem:e.3} 
Let $\alpha:F\to G$ be a morphism in $\Fun(\mc{J},\mc{S})$. Suppose that there exists an integer $b$ such that $\alpha(W):F(W)\to G(W)$ is $((n + 1)\dim(W)- b)$\=/connected for all $W$ in \mc{J}. Then $T_n(\alpha) : T_n F(W)\to T_n G(W)$ is ${((n + 1)\dim(W)-b + 1)}$\=/connected for all $W$.
\end{lemma}
It is important and not difficult to note that the conclusion of the lemma remains true even if the assumption on $\alpha(W)$ is only made for all $W$ with $\dim W\ge \kappa$ for some $\kappa\ge 0$. 

\begin{lemma}\label{prop:PnisoW}
Let $\alpha:F\to G$ be a natural transformation in $\Fun(\mc{J},\mc{S})$ such that the connectivity of the map $\alpha(W):F(W)\to G(W)$ is $\geq (n+1)\dim(W)-b$ for all $W$ in $ \mc{J}$ of dimension $\ge\kappa$. Then $P_n(\alpha): P_nF\to P_n$ is invertible.
\end{lemma}

\begin{proof} 
Under the assumptions on $\alpha$, Weiss shows in Lemma~\ref{lem:e.3} that the connectivity of the map $T_n(\alpha)$ is $\geq (n+1)\dim(W)-b+1$ for all $W$ of dimension $\ge \kappa-1$.
It follows by induction on $\ell$ that the connectivity of the map $T^{\ell}\alpha(W)$ is $\geq (n+1)\dim(W)-b+\ell$ for all $W$ of dimension $\ge \kappa-l$.
Hence the connectivity of the map $T^{\ell}(\alpha)(W)$  tends to infinity with $\ell$ for all objects of $\mc{J}$.
\end{proof}

\begin{proposition}\label{prop:Pn-zetan-WeissA}
The map $z_{n+1}:\sph(-)^{\star n+1} \to 1$  in $\Fun(\mc{J},\mc{S})$ is tidy. 
\end{proposition}

\begin{proof} 
The source of $z_{n+1}$ is compact by Corollary~\ref{cor:S-tothe-n+1-is-compact}.
If $W$ in \mc{J} is of dimension $m$, then $\sph(W)= S^{m-1}$ and
$\bigl(\sph(W)\bigr)^{\star n+1}=(S^{m-1})^{\star n+1}$ and its con\-nec\-ti\-vity is $(n+1)(m-2)+2n= (n+1)\dim(W)-2$. 
Hence $z_{n+1}$ satisfies the hypothesis of Lemma~\ref{prop:PnisoW} with $b=2$ and $\kappa=1$.
\end{proof}

\begin{definition}[{\cite[Definition 5.1]{Weiss95}}]
A functor $F:\mc{J} \to \mc{S}$ is {\it polynomial of degree $\leq n$} if the map $F\to T_n(F)$ is invertible. In other words, a functor $F:\mc{J} \to \mc{S}$ is polynomial of degree $\leq n$ if and only if it is $T_n$\=/closed. 
\end{definition}

\begin{theorem}[Weiss]\label{thm:orthog-tower} 
\emph{Theorem~\ref{thm:monoidal-localization}} applies to the orthogonal tower.
The functor $P_n$ defined above is a left exact reflector onto the subcategory of polynomial functors of degree $\le n$. 
\end{theorem}

\begin{proof}
Proposition~\ref{prop:Pn-zetan-WeissA} serves as input into Theorem~\ref{thm:monoidal-localization}. By Weiss' definition a $T_n$\=/closed object is exactly a polynomial functor of degree $\le n$. So the theorem is proved.
\end{proof}

\subsection{The orthogonal tower is a completion tower}
\label{subsec:ortho-tower-is-completion-tower}
The orthogonal tower $\hdots\to P_2\to P_1\to P_0$ is a tower of left exact localizations of the functor category $\Fun(\mc{J},\mc{S})$. 
If $\scr{A}_n$ denotes the class of $P_n$\=/equivalences in $\Fun(\mc{J},\mc{S})$, then Proposition~\ref{prop:enrichedsatgen} tells us that $\scr{A}_n$ is the left class of a fac\-torization system which is in fact a left exact modality. In particular, $\scr{A}_n$ is a congruence (see  Subsection~\ref{sec:associated-fs}).

In~\cite[Section 4.2]{ABFJ:product} we in\-tro\-duce, for any congruence $\scr{C}_0$ in a topos, its completion tower as a tower of left exact localizations whose $n$\=/th congruence $\scr{C}_n$ is obtained as the ${(n+1)}$\=/fold acyclic power $\scr{C}_n=\scr{C}_0\,\Box\ac\hdots\Box\ac\,\scr{C}_0$ of its $0$\=/th level. Here $-\Box\ac-$ denotes the acyclic product of congruences constructed in~\cite[Section 3.2]{ABFJ:product}. 
We will need the properties that the acyclic product $\scr{C}\,\Box\ac\,\scr{C}'$ of two congruences is again a congruence~\cite[Theorem 3.3.3]{ABFJ:product} 
and the formula~\cite[Corollary 3.3.8]{ABFJ:product} to compute $\scr{C}\,\Box\ac\,\scr{C}'$ in terms of generators for $\scr{C}$ and $\scr{C}'$, see Equation~\eqref{eq:lex-generator-formula}.
The proof that the orthogonal tower is a completion tower we result from the equalities $\scr{A}_n=\scr{C}_n$, for all $n\ge 0$. 

In any presheaf category, the representable functors form a dense subcategory. Since any object of \mc{J} is isomorphic to $\mb{R}^k$ for some $k$. The family of representable functors at $\mb{R}^k, k\ge 0$, form a dense subcategory of compact objects of $\Fun(\mc{J},\mc{S})$.
Thus, according to Proposition~\ref{prop:enrichedsatgen}, the congruence $\scr{A}_n$ is generated (as a saturated class) by the set
\[
S_n=\left\{\sigma_{n+1}^k=z_{n+1}\otimes\st(\mb{R}^k,-):\sph(-)^{\star n+1}\otimes \st(\mb{R}^k,-)\to \st(\mb{R}^k,-)\,|\, k\ge 0 \right\}\,.
\]
Since the Day product preserves colimits, Lemma~\ref{Zn-finite} gives:
\begin{align}
\begin{split}  \label{eq:nice-parallel}
   \sph(-)^{\star n+1} \,\otimes\, \st(\mb{R}^k,-) &= \left(\colim_{0\neq U\subset \mb{R}^{n+1}} \st(U, -) \right) \,\otimes\,   \st(\mb{R}^k,-) \\
       &= \colim_{0\neq U\subset \mb{R}^{n+1}} \Big( \st(U,-) \,\otimes\, \st(\mb{R}^k, -) \Big) \\
       &= \colim_{0\neq U\subset \mb{R}^{n+1}} \st( \mb{R}^k\oplus U,-)  
\end{split}
\end{align}
More precisely, Lemma~\ref{Zn-finite} identifies the map $\sigma_{n+1}^k$ with the ${(n+1)}$\=/st fiberwise join power $j_k\star_B\hdots\star_B j_k$. This map fits into a cartesian square 
\begin{equation}
\label{eq:Ganea-pullback-orthogonal-calc}
\begin{tikzcd}
\colim\limits_{0\neq U\subset \mb{R}^{n+1}} \st( \mb{R}^k\oplus U,-)  
\ar[rr] \ar[d,"{\sigma_{n+1}^k=j_k\star_B\hdots\star_B j_k}"'] 
&& F_{n+1}^k \ar[d,"{j_k^{\square n+1}}"] \\
B=\st(\mb{R}^k,-) \ar[rr,"\Delta"]
&& \st(\mb{R}^k,-)^{n+1} 
\end{tikzcd}
\end{equation}
by Lemma~\ref{lem:fiberwise-join-is-base-change}\,\ref{lem:fiberwise-join-is-base-change:4} applied objectwise. (The source $F_{n+1}^k$ of the join power is a relative fat wedge described in Lemma~\ref{lem:fiberwise-join-is-base-change}\,\ref{lem:fiberwise-join-is-base-change:3}.)
For $n=0$ and all $k$, we have
\[
\sigma_1^k=j_k:\st(\mb{R}^{k+1},-)\to\st(\mb{R}^k,-)
\]
and $S_0=\{j_k\,|\, k\ge  0\}$. Note that an $F:\mc{J}\to\mc{S}$ is $S_0$\=/local if and only if it is constant and the reflector $P_0$ can be identified with Weiss' description
\begin{equation}\label{eq:P0-is-F-at-Rinf}
   P_0F=\colim_{k\in\mb{N}} F(\mb{R}^k)=: F(\mb{R}^{\infty})\,. 
\end{equation}

\begin{remark}
An easy computation shows that the functor $i:\mb N\to \mc J$ sending $n$ to $\mb R^n$ and $n\leq n+1$ to $i_n$ is cofinal.
This shows that one can replace the colimit of \eqref{eq:P0-is-F-at-Rinf} by a colimit over $\mc J$, \ie that $P_0$ is simply the colimit functor $\Fun(\mc J,\mc S)\to \mc S$.
The cofinality also shows also that the category $\mc J$ is filtered, which gives another proof that $P_0$ is a left exact localization.
\end{remark}

At this point let us recall from~\cite{ABFJ:higher-sheaves,ABFJ:product} about congruences that, first, an acyclic class can be defined as a saturated class of maps which is closed under base change, and second, a congruence can be defined as an acyclic class which is closed under diagonals.
Any class of maps $\Sigma$ is contained in a smallest acyclic class $\Sigma\ac$, and a smallest congruence $\Sigma\cg$.
Since we have always $\Sigma\sat \subset\Sigma\ac \subset \Sigma\cg$, the formula $\scr{A}_0 = S_0\sat$ shows that we have also $\scr{A}_0 = S_0\ac = S_0\cg$ since $\scr{A}_0$ is a congruence.
This makes $S_0$ into a lex generator for the congruence $\scr{A}_0$ in the sense of \cite[Definition 2.4.4]{ABFJ:product}.
Consequently, by~\cite[Corollary 3.3.8]{ABFJ:product}, we get the following formula for the acyclic powers of $\scr{A}_0$:
\begin{equation}\label{eq:lex-generator-formula}
\scr{C}_n:=\underbrace{\ \scr{A}_0\,\Box\ac\hdots\Box\ac\,\scr{A}_0\, }_{n+1} =\big\{j_{k_1}\pp\hdots\pp j_{k_{n+1}}\,\big|\, k_1,\hdots,k_{n+1}\ge 0\big\}\ac\,.
\end{equation}

\begin{theorem}\label{thm:ortho-is-compl}
The orthogonal tower is a special case of a completion tower: $\scr{A}_n=\scr{C}_n$ for all $n\ge 0$.
\end{theorem}

\begin{proof}
Let us define an auxiliary congruence $\scr{B}_n=\big\{ j_k^{\Box n+1}\,\big|\, k\ge 0\big\}\cg$. Then:
\[
\scr{A}_n\subset\scr{B}_n\subset\scr{C}_n\,.
\]
The first inclusion holds since the generating maps of $\scr{A}_n=\{\sigma_{n+1}^k\,|\, k\ge 0 \}\sat$ are obtained from the generators of $\scr{B}_n$ by base change as exhibited by the cartesian square~\eqref{eq:Ganea-pullback-orthogonal-calc} and congruences are closed under base change. The second inclusion is clear since the generators of $\scr{B}_n$ are among those of $\scr{C}_n$.

To show the reverse inclusion we prove that the generators $j_{k_1}\pp\hdots\pp j_{k_{n+1}}$ of $\scr{C}_n$ from~\eqref{eq:lex-generator-formula} belong to $\scr{A}_n$. In other words, we need to explain that they are $P_n$\=/equivalences. This happens via a connectivity estimate. Let $W$ be a Euclidean vector space of dimension $m$. Then the fiber of the map $j_{\ell}:\st(\mb{R}^{\ell+1},W)\to\st(\mb{R}^{\ell},W)$ is a sphere of dimension $m-\ell-1$ as long as $\ell<m$. So $j_{\ell}$ is ${(m-\ell-2)}$\=/connected. Therefore, if $m>\max\{k_i\}$, the connectivity of the map $j_{k_1}\pp\hdots\pp j_{k_{n+1}}$ is
\[
\sum_{i=1}^{n+1}(m-k_i-2)+2n=(n+1)m-2-\sum_{i=1}^{n+1}k_i\,.
\]
Hence, $j_{k_1}\pp\hdots\pp j_{k_{n+1}}$ satisfies the assumptions of Lemma~\ref{prop:PnisoW} which implies that it is a $P_n$\=/equivalence. 
Thus, $\scr{C}_n\subset\scr{A}_n$ and the equality follows.
\end{proof}

\subsection{Blakers-Massey theorems}
\label{subsec:blakers-massey}

For any completion tower, there is an associated Blakers-Massey theorem and a ``dual'' version. As a consequence, we obtain new results for the orthogonal tower. The respective versions for the Goodwillie tower were proved as Theorems 3.4.1 and 3.4.2 in~\cite{ABFJ:gbm2}. 

\begin{corollary}\label{thm:orthog-BMT}
Consider in the category $\Fun(\mc{J},\mc{S})$ a pushout square
\[
\begin{tikzcd}
A \ar[r, "g"] \ar[d,"{f}"']& C \ar[d]   \\
B  \ar[r] & D
\end{tikzcd}
\]
where $f$ is a $P_m$\=/equivalence and $g$ is a $P_n$\=/equivalence. Then the gap map
\[
A\xto{(f,g)} B\times_D C
\]
is a  $P_{m+n+1}$\=/equivalence.
\end{corollary}

In the proof we are using the notation from the previous section: $\scr{A}_n$ are the $P_n$\=/equivalences. We also denote the acyclic powers simply as powers (since the cartesian product is never considered). 
\begin{proof}
By assumption $f\in\scr{A}_m$ and $g\in\scr{A}_n$. Congruences are closed under finite limits, hence $(\Delta f:A\to A\times_B A)\in\scr{A}_m=\scr{A}_0^{m+1}$ and $\Delta g\in\scr{A}_n=\scr{A}_0^{n+1}$ since the orthogonal tower is a completion tower by Theorem~\ref{thm:ortho-is-compl}. 
By the generalized Blakers-Massey \cite[Theorem 4.1.1]{ABFJ:gbm} we deduce
\[
(f,g)\in \bigl(\Delta f\pp \Delta g\bigr)\ac\subset (\scr{A}_0^{m+1}\pp \scr{A}_0^{n+1})\ac\subset\scr{A}_0^{m+n+2}
\]
proving the claim. (The inclusion to the right is actually an equality by~\cite[Theorem 3.3.3]{ABFJ:product}, but this is not really needed here.) 
\end{proof}

\begin{corollary}\label{thm:orthog-dual-BMT}
Consider in the category $\Fun(\mc{J},\mc{S})$ a pullback square
\[
\begin{tikzcd}
A \ar[r] \ar[d]& C \ar[d, "g"]   \\
B  \ar[r, "f"] & D
\end{tikzcd}
\]
where $f$ is a $P_m$\=/equivalence and $g$ is a $P_n$\=/equivalence. Then the cogap map
\[
B\sqcup_A C \to D
\]
is a  $P_{m+n+1}$\=/equivalence.
\end{corollary}

\begin{proof}
Similarly to the proof above, by~\cite[Theorem 3.5.1]{ABFJ:gbm} the cogap map is in $(\scr{A}_0^{m+1}\pp \scr{A}_0^{n+1})\ac\subset\scr{A}_0^{m+n+2}$.
\end{proof}

In fact, all the appropriate analogues of the statements in~\cite[Section 2.5]{ABFJ:gbm2} about stability and delooping results hold for orthogonal calculus as they hold for any completion tower in the sense of~\cite{ABFJ:product}. For example, it follows from Corollary~\ref{thm:orthog-BMT}, Corollary~\ref{thm:orthog-dual-BMT} that the category of functors that are ${(2n-1)}$\=/polynomial (\ie $P_{2n-1}F=F$), $n$\=/reduced (\ie $P_{n-1}F=1$), and pointed (equipped with a global section) is stable.

\section{Localizing module categories}
\label{sec:modules}

\subsection{Tidy Localizations of confined module categories}

\begin{definition}\label{def:QZ}
From a map $z:Z\to \one$ in a confined symmetric monoidal category \mc{V} we obtain by cotensoring a \mc{V}\=/functor 
\[
S:=\{Z,-\}:\mc{M} \to \mc{M}
\]
and $s=\{z,-\}:\Id\to S$.  
We repeat Construction~\ref{def:PZ} with $S$ (instead of $T$) in the context of a confined \mc{V}\=/module \mc{M}. We will denote by $Q:\mc{M} \to \mc{M}$ the colimit
\[
Q:=\colim \bigl( \Id\xto{s} S \xto{sS} S^2\xto{sS^2} S^3\xto{sS^3} S^4\to \hdots \bigr)\,.
\]
Even though $z$ is a map in \mc{V}, the resulting $Q$ is an endofunctor of \mc{M}.
\end{definition}

The analogue of Lemma~\ref{lem:Qisgood} for $S$, $s$, $Q$ and $q$ holds and the proof given there goes through because all the statements referenced there, \eg Lemma~\ref{fact:1}, Proposition~\ref{prop:Heine-rules}, Proposition~\ref{prop:filt-colim-good}, were proved in sufficient generality. We again point out that the theory of enriched higher categories, as developped in the papers~\cite{Heine:equiv-enriched-inf-cat-inf-cat-weak-action,Heine:bi-enriched} by Heine, is a necessary foundation.

Theorem~\ref{thm:module-localization} about localizing module categories is analogous to and a consequence of Theorem~\ref{thm:monoidal-localization} if we use the trick of putting a symmetric monoidal category and a closed module over it into a single symmetric monoidal category. This construction, explained just below, is analogous to a square zero extension from ordinary commutative algebra. 

Let us consider the category $\mc{V}\times \mc{M}$. Since limits and colimits are computed factorwise, for every objects $(V,M)$ in $\mc{V}\times \mc{M}$ we have
\begin{equation*} 
(V,M)=(V,0)\sqcup (0, M) \quad {\rm and }\quad
(V,M)=(V,1)\times (1, M)\ .
\end{equation*}
There are fully faithful inclusion functors 
\[
i_{\mc{V}}:\mc{V}\cong\mc{V}\times\{0\}\to\mc{V}\times\mc{M}
\]
and
\[
i_{\mc{M}}:\mc{M}\cong\{0\}\times\mc{M}\to\mc{V}\times\mc{M}
\]
and we will identify the categories \mc{V} and $\mc{M}$ with the respective full subcategories of $\mc{V}\times\mc{M}$.

\begin{proposition} \label{lem:square0construction} 
The category $\mc{V}\times\mc{M}$ has the structure of a symmetric monoidal closed category with the tensor product defined by
\begin{equation*}
(V_1,M_1)\otimes (V_2,M_2)=(V_1\otimes V_2, V_1\otimes M_2\sqcup V_2\otimes M_1)\ .
\end{equation*}
The unit is $\one:=(\one_{\mc{V}},0)$. 
Moreover,
\begin{equation*}
\big[(V_1,M_1), (V_2,M_2)\big]_{\mc{V}\times\mc{M}}=\big([V_1,V_2]\times [M_1,M_2], \{V_1,M_2\}\big)\,.
\end{equation*}
The symmetric monoidal category $\mc{V}\times\mc{M}$ is confined, if \mc{V} and $\mc{M}$ are confined.
\end{proposition}

The proof is left to the reader. 

Now let us consider the image of the map $z:Z\to\one$ in $\mc{V}\times\mc{M}$:
\[
i_{\mc{V}}(z)=\bigl((Z,0)\to(\one,0)\bigr)=(z,\id_0)\,.
\]
Following the recipe in Construction~\ref{def:PZ} we write down the associated cotensor
\begin{align*}
  \scr{T}(F):=[Z,F]_{\mc{V}\times\mc{M}}
\end{align*}
and
\begin{align*}
  \tau(F):=[i_{\mc{V}}(z),F]_{\mc{V}\times\mc{M}}: F=(V,M)\to \big[(Z,0),(V,M)\big]_{\mc{V}\times\mc{M}}
\end{align*}
for every $F=(V,M)$ in $\mc{V}\times\mc{M}$, and obtain an endofunctor $\scr{T}:\mc{V}\times\mc{M} \to \mc{V}\times\mc{M}$ together with a natural transformation $\tau: \Id \to \scr{T}$.
By construction we have
\begin{align*}
   \scr{T}(F) &= [Z,F]_{\mc{V}\times\mc{M}}=\big[(Z,0), (V,M)\big]_{\mc{V}\times\mc{M}}=\big([Z,V]_{\mc{V}}, \{Z, M\}\big)  \\
              &=(TV,SM)
\end{align*}
since $[0, M]=1$. Moreover
\begin{equation}\label{eq:tau-in-VM} 
   \tau(F) =[i_{\mc{V}}(z),F]_{\mc{V}\times\mc{M}}=\big([z,V], \{z, M\}\big) =(tV,sM)
\end{equation}
with $t:\Id\to T$ and $s:\Id\to S$ from Construction~\ref{def:PZ}, Definition~\ref{def:QZ}.
We also define an endofunctor $\scr{P}: \mc{V}\times\mc{M} \to \mc{V}\times\mc{M}$ together with a natural transformation $\pi: \Id\to \scr{P}$ by the colimit
\[
\scr{P}:=\colim \bigl(\ \Id\xto{\tau} \scr{T}\xto{\tau\scr{T}}  \scr{T}^2\xto{\tau\scr{T}^2}  \scr{T}^3\xto{\tau\scr{T}^3}  \scr{T}^4\to \hdots\ \bigr)\,.
\]
Then for every $F=(V,M)$ in \mc{V} we have 
\begin{equation} \label{eq:cP}
   \scr{P}(F)=( PV, QM)\quad {\rm and} \quad 
   \pi(F)    =(pV,qM) .
\end{equation}
With this construction in place we are ready to prove the following

\begin{theorem}\label{thm:module-localization}
Let \mc{V} be a confined symmetric monoidal category and let $\mc{M}$ be a confined \mc{V}\=/module.
Suppose that the map $z:Z\to \one$ in \mc{V} is tidy. 
Let $\mc{M}^Q$ denote the full subcategory of $Q$\=/closed objects of \mc{M}.
\begin{enumerate}[label={\rm (\roman*)}]
  \item\label{thm:module-localization:1} 
An object in $\mc{M}$ is $S$\=/closed if and only if it is $Q$\=/closed.
  \item\label{thm:module-localization:2}
The subcategory $\mc{M}^Q$ is \mc{V}\=/reflective, and the natural transformation $q:\Id\to Q$ is \mc{V}\=/reflecting into $\mc{M}^Q$. The reflector $Q:\mc{M}\to \mc{M}^Q$ is \mc{V}\=/left exact.
  \item\label{thm:module-localization:3}
The category $\mc{M}^Q$ is a closed $\mc{V}^P$\=/module with the action $\otimes_Q$ defined by letting $F\otimes_Q M:=Q(F\otimes M)$ for $F$ in $\mc{V}^P$ and $M$ in $\mc{M}^Q$. 
  \item\label{thm:module-localization:4}
The closed $\mc{V}^P$\=/module $\mc{M}^Q$ is confined and $Q:\mc{M}\to\mc{M}^Q$ is confined.
Every compact object of $\mc{M}^Q$ is a retract of an object in $Q(\comp(\mc{M}))$. The subcategory $Q(\comp(\mc{M}))$ is dense in $\mc{M}^Q$.
\end{enumerate}
\end{theorem}

\begin{proof} 
We want to apply Theorem~\ref{thm:monoidal-localization} to the confined symmetric monoidal category $\mc{V}\times\mc{M}$ but we need to first show that the map $i_{\mc{V}}(z)=(z,\id_0)$ is tidy. 
Observe that $i_{\mc{V}}(Z)=(Z,0)$ and $i_{\mc{V}}(\one)=(\one,0)$ are clearly compact in $\mc{V}\times\mc{M}$ and
\[
\scr{P}(i_{\mc{V}}(z))=\big(P(z),Q(\id_0)\big)
\]
is invertible if and only if $P(z)$ is invertible. But $P(z)$ is invertible by the assumption that $z$ in \mc{V} is tidy. So $i_{\mc{V}}(z)=(z,\id_0)$ is tidy.
Now we just need to read off, what Theorem~\ref{thm:monoidal-localization} applied to $\mc{V}\times\mc{M}$ says for \mc{M} viewed as a subcategory via $i_{\mc{M}}:\mc{M}\to\mc{V}\times\mc{M}$.

\smallskip
\noindent\ref{thm:module-localization:1}
Theorem~\ref{thm:monoidal-localization}\,\ref{thm:monoidal-localization:1} states for $\mc{V}\times\mc{M}$ that \scr{T}\=/closed is equivalent to being \scr{P}\=/closed, and for \mc{V} that $T$\=/closed is equivalent to being $P$\=/closed. 
It follows from~\eqref{eq:tau-in-VM} that an object $(V,M)$ in $\mc{V}\times\mc{M}$ is \scr{T}\=/closed if and only if $V$ is $T$\=/closed and $M$ is $S$\=/closed.
Similarly, from~\eqref{eq:cP}, $(V,M)$ is \scr{P}\=/closed if and only if $V$ is $P$\=/closed and $M$ is $Q$\=/closed. Hence, $M$ in \mc{M} is $S$\=/closed if and only if it is $Q$\=/closed.

Let $(\mc{V}\times\mc{M})^{\scr{P}}$ denote the full subcategory of \scr{P}\=/closed objects of $\mc{V}\times\mc{M}$. From \ref{thm:module-localization:1}, we know that the standard inclusion factors as $(\mc{V}\times\mc{M})^{\scr{P}}\cong\mc{V}^P\times\mc{M}^Q\subset\mc{V}\times\mc{M}=\mc{V}\times\mc{M}$.

\smallskip
\noindent\ref{thm:module-localization:2}
Theorem~\ref{thm:monoidal-localization}\,\ref{thm:monoidal-localization:3} applied to $\mc{V}\times\mc{M}$ states that $\pi:\Id\to\scr{P}$ is \mc{V}\=/reflecting onto $(\mc{V}\times\mc{M})^{\scr{P}}$: 
\begin{align*} 
   \big([PV_1,V_2]\times & [QM_1,M_2], \{PV_1,M_2\}\big)=
   \big[\scr{P}(V_1,M_1),(V_2,M_2)\big]_{\mc{V}\times\mc{M}}
   \xto{\cong} \\
   & \big[(V_1,M_1),(V_2,M_2)\big]_{\mc{V}\times\mc{M}} = \big( [V_1,V_2]\times [M_1,M_2],\{V_1,M_2\} \big)
\end{align*}
for all $(V_1,M_1)$ in $\mc{V}\times\mc{M}$ and $(V_2,M_2)$ in $(\mc{V}\times\mc{M})^{\scr{P}}=\mc{V}^P\times\mc{M}^Q$. 
When $V_1=0$, for all $M_1$ in \mc{M} and $M_2$ in $\mc{M}^Q$, we have 
\begin{align*}
    \bigl( [QM_1,M_2], 1\bigr)& =\bigl([P(0),0]\times [QM_1,M_2], \{P(0),M_2\}\bigr) \\
                            &=\bigl( [0,V_2]\times [M_1,M_2],\{0,M_2\} \bigr) \\
                            & =\bigl([M_1,M_2], 1\bigr)\,.
\end{align*}
So clearly $q:\Id\to Q$ is \mc{V}\=/reflecting \mc{M} onto $\mc{M}^Q$. The fact, that for all $M_1$ the object $QM_1$ is in $\mc{M}^Q$, was already clear since $\scr{P}(0,M_1)=(0,QM_1)$ is in $(\mc{V}\times\mc{M})^{\scr{P}}=\mc{V}^P\times\mc{M}^Q$.

The endofunctor $\scr{P}=(P,Q)$ is left exact and preserves compact cotensors by Theorem~\ref{thm:monoidal-localization}\,\ref{thm:monoidal-localization:3}. Since (co-)limits in $\mc{V}\times\mc{M}$ are computed separately in \mc{V} and \mc{M}, it follows that $Q$ preserves finite limits. Hence:
\begin{align*}
  \big[0, Q\{V_1,M_2\}\big]_{\mc{V}\times\mc{M}}  &=\scr{P}[(V_1,0),(0,M_2)]_{\mc{V}\times\mc{M}} =[(V_1,0),\scr{P}(0,M_2)]_{\mc{V}\times\mc{M}}  \\
         &= [(V_1,0),(0,QM_2)]_{\mc{V}\times\mc{M}} =[0,\{V_1,QM_2\}]_{\mc{V}\times\mc{M}} \,,
\end{align*}
for all compact $V_1$ in \mc{V} and all $M_2$ in \mc{M}. So $Q$ also preserves compact cotensors and is therefore \mc{V}\=/left exact.

\smallskip
\noindent\ref{thm:module-localization:3}
According to Theorem~\ref{thm:monoidal-localization}\,\ref{thm:monoidal-localization:4} is the category $(\mc{V}\times\mc{M})^{\scr{P}}$ a symmetric monoidal $\mc{V}^P$\=/module with tensor
\begin{align*}
   (V_1,M_1)\otimes_{\scr{P}} (V_2,M_2) & =\scr{P}\bigl((V_1,M_1)\otimes (V_2,M_2)\bigr) \\
      & = \scr{P}(V_1\otimes V_2, V_1\otimes M_2\sqcup V_2\otimes M_1) \\
      & = \bigl((P(V_1\otimes V_2), Q(V_1\otimes M_2\sqcup V_2\otimes M_1)\bigr) \,,
\end{align*}
for all $V_1,V_2$ in $\mc{V}^P$ and $M_1,M_2$ in $\mc{M}^Q$. Setting $V_2$ and $M_1$ equal to $0$ one obtains an action of $\mc{V}^P$ on $\mc{M}^Q$ given by 
\[
V_1\otimes_Q M_2 = Q(V_1\otimes M_2)\,,
\]
as claimed and $Q$ preserves the tensor action. 

\smallskip
\noindent\ref{thm:module-localization:4}
We know from~\ref{thm:module-localization:2} that the category $\mc{M}^Q$ is a reflective subcategory of \mc{M}. So it is cocomplete and the localization functor $Q:\mc{M}\to \mc{M}^Q$ is cocontinous. 
As an endofunctor, $\scr{P}=(P,Q)$ is docile by Lemma~\ref{lem:Qisgood} applied to $\mc{V}\times\mc{M}$. But colimits are computed separately in \mc{V} and \mc{M}. Thus, $Q$ commutes with filtered colimits. In~\ref{thm:module-localization:3} we showed that $Q$ is \mc{V}\=/left exact, hence it is a docile endofunctor of \mc{M}. Thus the inclusion $\mc{M}^{Q}\subset \mc{M}$ preserves filtered colimits. By Proposition~\ref{prop:confinedbasic2}, the left adjoint localization functor $Q:\mc{M}\to \mc{M}^Q$ is confined.
In particular, $Q(\comp(\mc{M})) \subset\comp(\mc{M}^Q)$.

From this, it follows that $Q(\comp(\mc{M}))$ is dense in $\mc{M}^Q$, that $\mc{M}^Q$ is $\omega$\=/presentable and that every object in $\comp(\mc{M}^Q)$ is a retract of an object in $Q(\comp(\mc{M}))$ exactly as in the proof of Theorem~\ref{thm:monoidal-localization}\,\ref{thm:monoidal-localization:5}.

It remains to show that the closed \mc{V}\=/category $\mc{M}^Q$ is confined. We have
\begin{align*}
    P(\comp(\mc{V}))\otimes_{Q} Q(\comp(\mc{M})=Q\bigl(\comp(\mc{V})\otimes \comp(\mc{M})\bigr)\subset Q(\comp(\mc{M})) \subset \comp(\mc{M}^{Q})\,.
\end{align*}
by the definition of $-\otimes_Q-$, by \mc{M} being confined as a closed \mc{V}\=/module, and by~\ref{thm:module-localization:3}. Also by~\ref{thm:module-localization:3}, the category $Q(\comp(\mc{V}))$ is dense in $\mc{M}^Q$. Lemma~\ref{sufficientforconfinedtensor} allows for an adaption to the module case with the same proof. It follows from this adaption that the closed \mc{V}\=/category $\mc{M}^Q$ is confined.  
\end{proof}

\begin{lemma} \label{lem:enrichedsatgen}
Let $z:Z\to \one$ be tidy. In $\mc{M}$ let $\mc{L}$ be the class of $Q$\=/equivalences and let $\mc{R}$ is the class of $Q$\=/closed maps. Write $\Sigma=\{Z\otimes M \xto{z\otimes M} M \,|\, M\ \text{compact}\,\}$.
\begin{enumerate}
   \item 
The pair $(\mc{L}, \mc{R})$ is a left exact modality in $\mc{M}$.
   \item 
$\Sigma^{\perp}=\mc{R}$ and $\mc{L}=\Sigma\sat=\mbox{}^{\perp}(\Sigma^{\perp})$
\end{enumerate}
\end{lemma}

The proof of this lemma is the same as the proof of Proposition~\ref{prop:enrichedsatgen}.

\begin{proposition}\label{prop:localisationmodule} 
Let \mc{V} be a confined symmetric monoidal category and let $Z\to\one$ be tidy. 
Let $\phi:\mc{M}\to \mc{N}$ be a confined functor of confined \mc{V}\=/modules.
Then the functor $\phi$ takes $Q$\=/equivalences in $\mc{M}$ to $Q$\=/equivalences in $\mc{N}$. 
Its right adjoint $\phi_{*}$ takes $Q$\=/closed objects in $\mc{N}$ to $Q$\=/closed objects in $\mc{M}$, and the squares 
\begin{equation*} 
   \begin{tikzcd}[bo column sep=large]
   \mc{M} \ar[r, "\phi"] \ar[d, "Q"']   & \mc{N} \ar[d, "Q"] \\
   \mc{M}^Q  \ar[r, "{\phi^Q}" ]   & \mc{N}^Q 
   \end{tikzcd}
 \quad \quad \quad  
   \begin{tikzcd}[bo column sep=large]
   \mc{M}  & \ar[l, "{\phi_{*}}"']   \mc{N}   \\
   \mc{M}^Q \ar[u, "{\mathrm{inc}}"]   &\ar[l, "{\phi_{*}|}"' ]   \mc{N}^Q \ar[u, "{\mathrm{inc}}"']
   \end{tikzcd}
\end{equation*}
commute, where $\phi^Q(M):=Q(\phi(M))$ for every $M$ in $\mc{M}^Q$ and $\phi^{*}|$ is the restriction of $\phi^{*}$.
\end{proposition}

\begin{proof} 
We write $\mc{L}_1, \mc{L}_2$ for the classes of $Q$\=/equivalences in $\mc{M}$ and $\mc{N}$ with $\mc{R}_2$ the corresponding right class of $Q$\=/closed maps. Similarly, we let $\Sigma_1=\{ z\otimes M \,|\, M\text{ in }\comp(\mc{M})\}$ and $\Sigma_2=\{ z\otimes N \,|\, N\text{ in }\comp(\mc{N})\}$. From Lemma~\ref{lem:enrichedsatgen} we know $\mc{L}_1=\Sigma_1\sat$ and $\mc{L}_2=\Sigma_2\sat$.

Since $\phi$ is a morphism of \mc{V}\=/modules, we have $\phi(z\otimes M)=z\otimes \phi(M)$. Since $\phi$ is confined, $\phi(M)$ is a compact object. It follows $\phi(\Sigma_1)\subset \Sigma_2 \subset \mc{L}_2$ or equivalently $\Sigma_1\subset \phi^{-1}(\mc{L}_2)$. But $\phi$ is cocontinuous. Therefore $\phi^{-1}(\mc{L}_2)$ is saturated and we have $\mc{L}_1=\Sigma\sat\subset \phi^{-1}(\mc{L}_2)$ or equivalently $\phi(\mc{L}_1)\subset \mc{L}_2$. This yields the first claim.
 
Now let $g$ be $Q$\=/closed in \mc{N}. For $\phi_{*}(g)$ to be $Q$\=/closed, it suffices by Lemma~\ref{lem:enrichedsatgen} that we have $f\perp \phi_{*}(g)$ for every $f\in  \mc{L}_2$. Since the functor $\phi_{*}$ is right adjoint to $\phi$, this is equivalent to $\phi(f)\perp g$ which is true, since $\phi(\mc{L}_1)\subset \mc{L}_2=\mbox{}^{\perp}\mc{R}_2$ by Lemma~\ref{lem:enrichedsatgen}.

Consequently, there is an induced functor $\phi_{*} |:\mc{N}^Q\to \mc{M}^Q$.
The right hand square above then clearly commutes. It is also easy to verify that the functor $\phi^Q:\mc{M}^Q\to \mc{N}^Q$ is left adjoint to $\phi_{*} |$ and so the left hand square above commutes.
\end{proof}

\subsection{Applications to module categories}

The first application of the module version of our localization technique in Theorem~\ref{thm:monoidal-localization} is to construct the Goodwillie tower in the category $\Fun(\mc{C},\mc{S})$ for any small category \mc{C} possessing finite colimits and a terminal object. 
The idea is that such a $\mc C$ admits an action of $\Fin$ equipped with the join product and this action can be extended into an action of $\Fun(\Fin,\mc{S})$ on $\Fun(\mc{C},\mc{S})$.

\medskip
To construct the action $-\star-:\Fin\times\mc{C}\to\mc{C}$ note first that $\Fin$ acts on any finitely cocomplete category:
\[
K\times C := \bigsqcup_{K} C\,.
\]
This is a tensor action in the ${(\infty,1)}$\=/categorical sense and this is proved in~\cite[Section 2.4.3]{HA}. Note that $1\times C=C$.
Using the terminal object of \mc{C} we can promote this action to a join. We define $K\star C$ as the pushout:
\[
  \begin{tikzcd}
     K\times C \ar[r,"{\pr_C}"] \ar[d,"{\pr_K}"'] & C \ar[d] \\
     K \ar[r]  & K\star C  
  \end{tikzcd}
\]   
Here the maps $\pr_C$ and $\pr_K$ are induced by the maps $C\to 1$ and $K\to 1$.

In this way we obtain the action $-\star-:\Fin\times\mc{C}\to\mc{C}$. Now we deduce from Example~\ref{ex:Day-action-from-action}, that the functor category $\mc{M}=\Fun(\mc{C},\mc{S})$ with its Day convolution product derived from this join action is a confined \mc{V}\=/module where $\mc{V}=\Fun(\Fin,\mc{S})$ with its Day convolution product derived from the join on $\Fin$. 

In Subsection~\ref{subsection:unpointed-Goodwillie} we showed that the map $z:\Id^{\star n+1}\to 1$ in $\Fun(\Fin,\mc{S})$ is tidy and for varying $n\ge 0$ produces the $n$\=/th stage of the Goodwillie tower. It is a consequence of Theorem~\ref{thm:module-localization} that this is now also the case in $\Fun(\mc{C},\mc{S})$. All we need to check is that Goodwillie's endofunctor $T_n$ coincides with the one here. But this is clearly the case since
\[
\{\Id^{n+1},F\}=\left\{\colim_{\emptyset\neq U\subset\ul{n+1}} \map(U,-), F \right\}=\lim_{\emptyset\neq U\subset\ul{n+1}} F(U\star -)
\]
by Lemma~\ref{lem:calculate-zeta2}, where $U\star -$ is the action of $\Fin$ on \mc{C}.

\begin{theorem}[Goodwillie] \label{thm:Goodwillie-all}
Let \mc{C} be a small finitely cocomplete category with a terminal object.
Then the tidy map $\Id^{\star n+1}\to 1$ in $\Fun(\Fin,\mc{S})$ yields Goodwillie's reflector $P_n$ in the category $\Fun(\mc{C},\mc{S})$.
\end{theorem}

When $\mc C=\Fin_\ast$ this provides another proof of Theorem~\ref{thm:pointed-Pn-reflects-onto-Tn}.

\subsection{The image of a tidy map via a confined functor} 

Consider a confined symmetric monoidal functor $\phi: \mc{V} \to \mc{E}$ between confined symmetric monoidal categories. In this situation Example~\ref{exam:conf-sym-mon-fun-yields-conf-mod} states that \mc{E} becomes a confined \mc{V}\=/module. Also $\phi(Z)$ is compact in \mc{E}.
Now we define $S:=[\phi(Z),-]:\mc{E}\to \mc{E}$, $s:=[\phi(z),-]=\{z,-\}:\Id \to S$, $Q=\colim_n{S^n}$ and $q:\Id \to Q$ as in Construction~\ref{def:PZ}. 
Note that $S$ can be viewed as an \mc{E}-functor $S=[\phi(Z),-]$ or as a \mc{V}\=/functor $S=\{Z,-\}$. Accordingly, $s:=[\phi(z),-]=\{z,-\}:\Id \to S$ is a natural transformation that is both \mc{V}\=/enriched and \mc{E}\=/enriched, and the same is true for the resulting $Q$. 
The functor $Q$ is docile both as a \mc{V}\=/functor and as an $\mc{E}$\=/functor by Lemma~\ref{lem:Qisgood} as it preserves cotensors by compact objects in \mc{V} and \mc{E}.

\begin{theorem} \label{thm:imagetidy} 
A confined symmetric monoidal functor $\phi: \mc{V} \to \mc{E}$ between confined symmetric monoidal categories takes a tidy map $z:Z\to \one$ in $\mc{V} $ to a tidy map $\phi(z):\phi(Z)\to \phi(\one)=\one$ in $ \mc{E}$.
\end{theorem}

\begin{proof} 
Taking $Q$ as an endofunctor of the closed \mc{V}\=/module \mc{E}, Theorem~\ref{thm:module-localization}\,\ref{thm:module-localization:2} applies and the natural transformation $q:\Id \to Q$ reflects \mc{E} onto the subcategory $\mc{E}^Q$ of $Q$\=/closed objects. So now we know already that the localization in \mc{E} exists, is confined and \mc{V}\=/left exact. Since $\{z,-\}=[\phi(z),-]$, the localization $Q$ is constructed according to Construction~\ref{def:PZ}, but we do not yet know that $\phi(z)$ is tidy. (In particular, we do not yet know that $Q$ is symmetric monoidal. But this will follow, once tidyness is proved.) It remains to work somewhat in reverse. 
We claim that $Q(\phi(z))$ is invertible.

\smallskip
\noindent Step 1: Take a compact object $K$ and a $Q$\=/closed object $N$ in \mc{E}. Then the triangle
\[
\begin{tikzcd}
       & {[K,N]} \ar[dl, "{q[K,N]}"']\ar[dr, "{[K,qN]}", "\cong"'] & \\
     {Q[K,N]} \ar[rr, "{\gamma(K,N)}"',"\cong"] && {[K,QN]}
\end{tikzcd}
\]
commutes. Since $Q$ is a docile \mc{E}-functor, the coassembly map is an isomorphism. The right hand map is an isomorphism since $N$ is $Q$\=/closed. It implies that the left hand map $q[K,N]$ is invertible. Hence $[K,N]$ is $Q$\=/closed for every compact $K$ in \mc{E}. 

\smallskip
\noindent Step 2: Now let $M$ be another object in \mc{E} and consider the map $[qM,N]$. We have a commuting diagram
\[
\begin{tikzcd}
       \map(K, [QM,N]) \ar[d, "\cong"']\ar[rrr, "{\map(K, [qM,N])}"] &&& \map(K, [M,N]) \ar[d, "\cong"'] \\
       \map(QM, [K,N]) \ar[rrr, "{\map(qM, [K,N])}"] &&& \map(M, [K,N]) 
\end{tikzcd}
\]
where the vertical maps are invertible. By Step 1 the object $[K,N]$ is $Q$\=/closed and so the lower map is an isomorphism. Thus the top map is also an isomorphism. Since \mc{E} is $\omega$\=/presentable, by Lemma~\ref{Whiteheadcolimit17} the map $[qM,N]$ is invertible.

\smallskip
\noindent Step 3: Now we copy the proof of Lemma~\ref{lem:tensorlocalnew}. For any map $u:M\to M'$ in \mc{E} the square
\begin{equation*} 
   \begin{tikzcd}[bo column sep=large]
   {[M,N]} && \ar[ll, "{[qM,N]}", "\cong"']   [QM,N] \\
   {[M',N]} \ar[u, "{[u,N]}"]   && \ar[ll, "{[qM',N]}", "\cong"'] \ar[u, "{[Q(u),N]}"']  [QM',N]  
   \end{tikzcd}
\end{equation*}   
commutes. The horizontal maps are invertible by Step 2. So the map $[u,N]$ is invertible if and only if the map $[Q(u),N]$ is invertible and, letting $N$ vary in all $\mc{E}^Q$, if and only if $Q(u)$ is invertible.

Step 4: $Q(\phi(z))$ is invertible if and only if the map $[\phi(z),N]=\{z,N\}=sN$ is invertible for all $Q$\=/closed $N$. 
Now Lemma~\ref{lem:Tn-local-Pn-local-mapsA}\,\ref{lem:Tn-local-Pn-local-mapsA:3} applies and states that $N$ is also $S$\=/closed. And the definition of being $S$\=/closed is that the map $sN$ is an isomorphism. So the theorem follows from Step 3.
\end{proof}

Let $\sigma:\mc S \to \mc S_\ast$ be the suspension functor, viewed as pointed by the north pole.
We consider $\mc S$ is equipped with the join monoidal structure and $\mc S_\ast$ with the smash monoidal structure.
\begin{lemma}
The functor $\sigma$ is symmetric monoidal.
\end{lemma}
\begin{proof}
By definition of the join product, the symmetric monoidal category $(\mc S,\star,\emptyset)$ is a submonoidal category of the arrow category with the pushout product $(\mathrm{Arr}(\mc S), \pp,\emptyset\to 1)$ where the embedding is sending a space $A$ to the arrow $A\to 1$.
The cofiber functor $\mathrm{Arr}(\mc S)\to \mc S_\ast$ is a reflection into the subcategory spanned by arrows $1\to A$.
The pushout product ${(1\to A)\pp(1\to B)}$ of two pointed objets is the arrow $A\vee B\to A\times B$, whose cofiber is the smash $1\to A\wedge B$.
We leave the reader to check that the cofiber functor is monoidal for the smash product in $\mc S_\ast$.
By composition, we get a symmetric monoidal functor $\mc S \to \mathrm{Arr}(\mc S)\to \mc S_\ast$, whose values on $A$ is the suspension $\Sigma A$.
\end{proof}

The functor $\sigma$ restricts to a symmetric monoidal functor between the monoidal subcategories of finite spaces
$\sigma:\Fin \to \Fin_\ast$.
By Example~\ref{ex:sym-mon-conf} we get a confined symmetric monoidal functor
$\sigma_!:\Fun(\Fin,\mc S)\to \Fun(\Fin_\ast,\mc S)$.
We can apply Theorem~\ref{thm:imagetidy} to get a new proof that the map $1\to \Id_\circ$ of is tidy.

\begin{corollary}
\label{cor:transport-tidy}
The symmetric monoidal functor $\sigma_!:\Fun(\Fin,\mc S)\to \Fun(\Fin_\ast,\mc S)$ is confined and the image of the tidy map $\Id \to 1$ is the map $1\to \Id_\circ$, which is then tidy.
\end{corollary}
\begin{proof}
The proof that $\sigma_!$ is confined is given in Corollary~\ref{cor:shriek-confined}.
By Yoneda, the map $\Id \to 1$ corresponds to the map $0\to 1$ in $\Fin$ and the map $1\to \Id_\ast$ corresponds to the map $S^0\to 1$ in $\Fin_\ast$.
The latter is the suspension of the former, thus 
$1\to \Id_\ast = \sigma_!(\Id \to 1)$ and its tidyness follows from Theorem~\ref{thm:imagetidy}.
\end{proof}

\end{document}